\documentclass[11pt]{article}
\usepackage{psfrag,graphicx,amsmath,amsthm,amssymb,color}

\newtheorem{remark}{Remark}[section]
\newtheorem{theorem}{Theorem}[section]
\newtheorem{lemma}{Lemma}[section]
\numberwithin{equation}{section}

\newcommand{\aver}[1]{\left\{\!\!\left\{#1\right\}\!\!\right\}}

\newcommand{\jump}[1]{\left[\!\left[#1\right]\!\right]}
\newcommand{\bigjump}[1]{\left[\!\!\left[#1\right]\!\!\right]}

\begin{document}

\title{{\it A Priori} Error Estimates for Some Discontinuous Galerkin \\
Immersed Finite Element Methods
\footnote{This work is partially supported by NSF grant DMS-1016313} }
\date{}

\author{
Tao Lin
\footnote{Department of Mathematics, Virginia Tech, Blacksburg, VA 24061, tlin@math.vt.edu}, \quad
Qing Yang\footnote{School of Mathematical Science, Shandong Normal University, Jinan 250014, P. R. China},\quad
Xu Zhang\footnote{Department of Mathematics, Purdue University, West Lafayette, IN, 47907, xuzhang@purdue.edu} }

\maketitle

\abstract
{In this paper, we derive {\it a priori} error estimates for a class of interior penalty discontinuous Galerkin (DG) methods using immersed finite element (IFE) functions for a classic second-order elliptic interface problem. The error estimation shows that these methods can converge optimally in a mesh-dependent energy norm. The combination of IFEs and DG formulation in these methods allows local mesh refinement in the Cartesian mesh structure for interface problems. Numerical results are provided to demonstrate the convergence and local mesh refinement features of these DG-IFE methods.
}
\section{Introduction}
Let $\Omega$ be a rectangular domain in $\mathbb{R}^2$, and let $\Gamma\subset\Omega$ be a smooth curve separating $\Omega$ into two sub-domains $\Omega^-$ and $\Omega^+$ with $\Omega^-\cap\Omega^+=\emptyset$ (see the first plot in Figure \ref{fig: domain}).
We consider the following typical elliptic interface problem
\begin{eqnarray}
-\nabla\cdot(\beta\nabla u)=f,&\text{in}&\Omega^+\cup\Omega^-, \label{eq: pde}\\
  u = 0, &\text{on}&\partial\Omega,\label{eq: bc}
\end{eqnarray}
where the diffusion coefficient $\beta$ is a positive piecewise constant function:
\begin{equation}\label{eq: coef jump}
  \beta(X)=\left\{
\begin{array}{ll}
\beta^-,& X\in\Omega^-,\\
\beta^+,& X\in\Omega^+.
\end{array}
\right.
\end{equation}
According to the conservation laws, the following jump conditions are required on the interface:
\begin{eqnarray}
 \jump{u}_{\Gamma}&=&0, \label{eq: jump1}\\
 \bigjump{\beta\frac{\partial u}{\partial \mathbf{n}}}_{\Gamma}&=&0.\label{eq: jump2}
\end{eqnarray}

Interface problems arise in many applications where mathematical simulations are carried out in a domain containing multiple materials. The elliptic interface problem \eqref{eq: pde} - \eqref{eq: jump2} considered in this article appears frequently because the involved differential equation captures many basic physical phenomenons.
A wide variety of numerical methods have been developed for interface problems, among which the finite element methods are advantageous for their capability to handle simulation domains with complicated geometry. It is well-known that conventional finite element methods generally require the mesh to fit the interface geometry
(see the second plot in Figure \ref{fig: domain}); otherwise, the convergence cannot be guaranteed \cite{Babuska_Osborn_Bad_FEM,JBramble_JKing_FEM_Interface,ZChen_JZou_FEM_Elliptic}. However, for a problem with a complicated material interface, constructing a satisfactory body-fitting mesh is often costly, and this burden becomes more severe if the simulation involves a moving interface \cite{XHe_TLin_YLin_XZhang_Moving_CNIFE, TLin_YLin_XZhang_IFE_MoL, Rangarajan_Lew_Universal_Mesh} because the mesh has to be generated repeatedly according to each interface location to be considered. In addition, some simulations, such as the particle-in-cell (PIC) method \cite{Birdsall_Langdon_Plasma_Physics,Kafafy_Wang_Whole_Ion_HG_IFE_PIC,Wang_He_Cao_IEEE}, can be carried out more efficiently on structured/Cartesian meshes.
Due to these reasons, a wide variety of numerical methods based on Cartesian meshes have been developed. For an overview of these methods, we refer to
\cite{Osher_Ghost_Fluid_1999, Hansbo_Hansbo_FEM_Elliptic, Hou_Wu_Multiscale_FE_Elliptic, ZLi_KIto_IIM, Moes_Dolbow_Belytschko_XFEM, CPeskin_IBM_2002} and the references therein.

Immersed finite element (IFE) methods were recently introduced for solving interface problems. The main feature of IFE methods is that they can use meshes independent of the interface location, \emph{i.e.}, they allow interface to cut through the interior of elements in a mesh (see the last two plots in Figure \ref{fig: domain}). Hence, Cartesian (triangular or rectangular) meshes can be preferably employed in IFE methods to solve interface problems. We refer the readers to \cite{
Chou_Kwak_Wee_IFE_Triangle_Analysis,
Kwak_Wee_Chang_Broken_P1_IFE,
ZLi_IIM_FE,
ZLi_TLin_YLin_RRogers_linear_IFE,
ZLi_TLin_XWu_Linear_IFE} for more features about IFE methods based on triangular meshes, and
\cite{
XHe_TLin_YLin_Bilinear_Approximation,
XHe_TLin_YLin_Convergence_IFE,
TLin_YLin_RRogers_MRyan_Rectangle,
SVallaghe_TPapadopoulo_TriLinear_IFE}
for IFE methods based on rectangular meshes. We note that IFE methods in these literatures are applied in the continuous Galerkin formulation.

\par
\begin{figure}[htb]\label{fig: domain}
\centering
\setlength{\unitlength}{1.5mm}
\includegraphics[totalheight=1.2in]{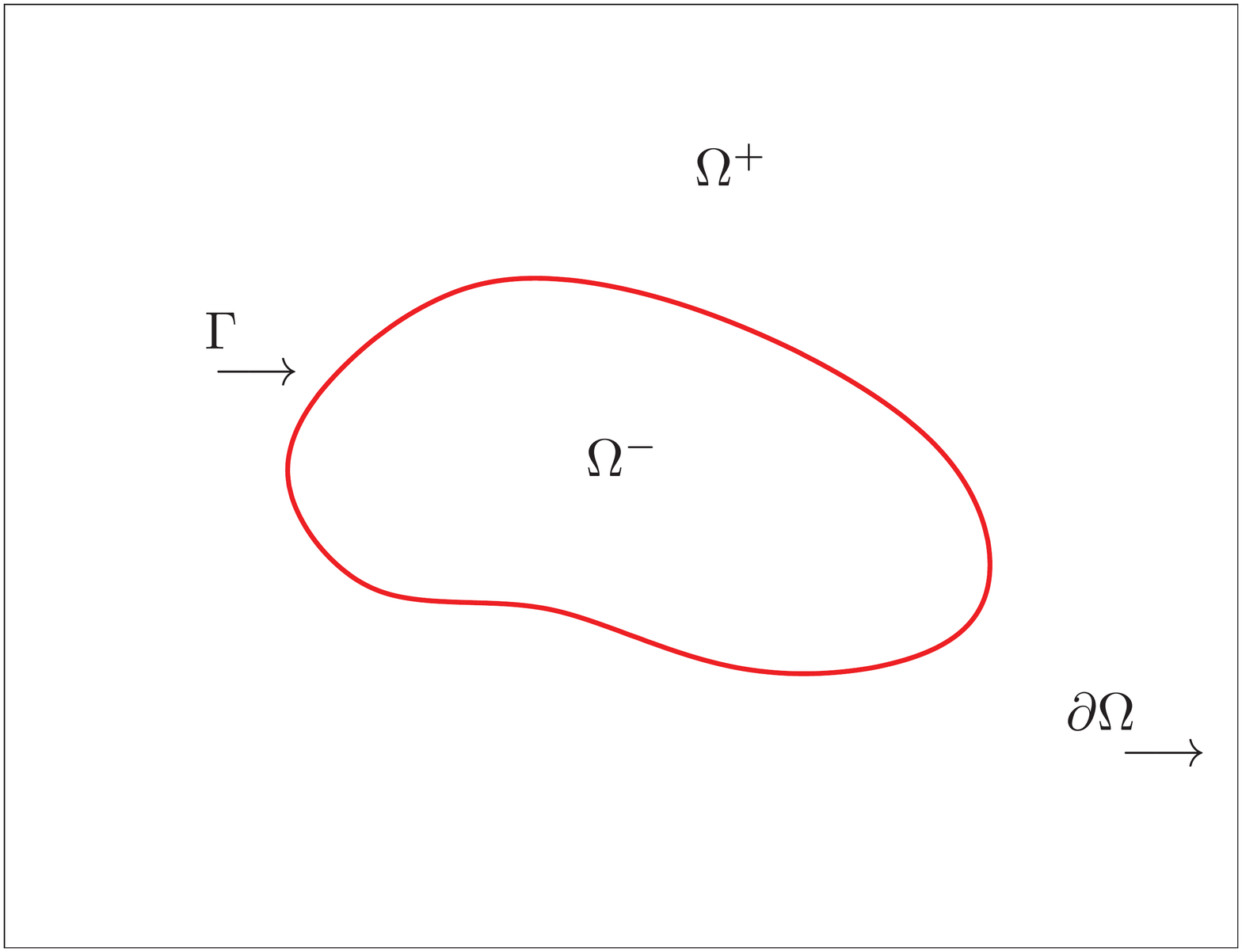}
\includegraphics[totalheight=1.2in]{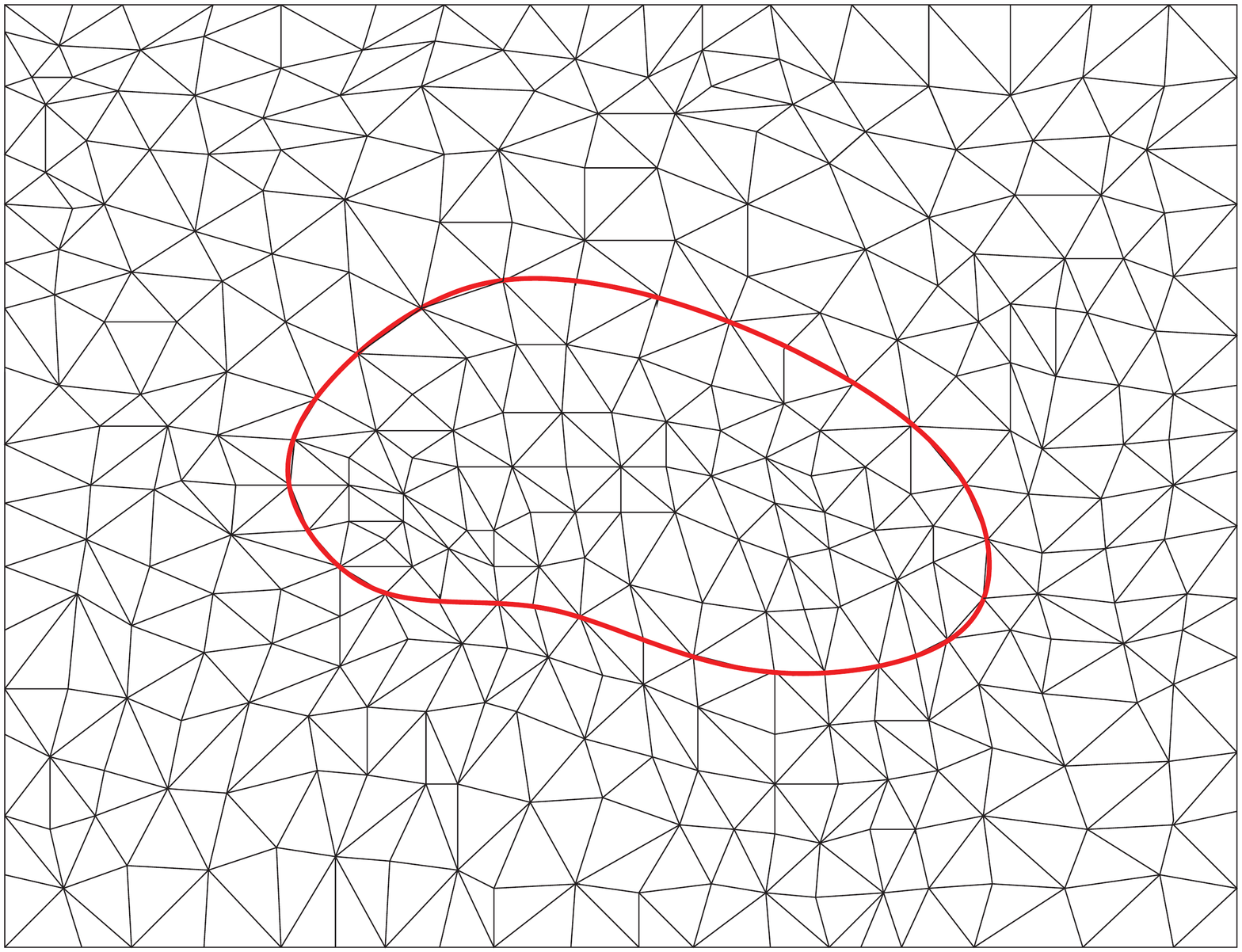}
\includegraphics[totalheight=1.2in]{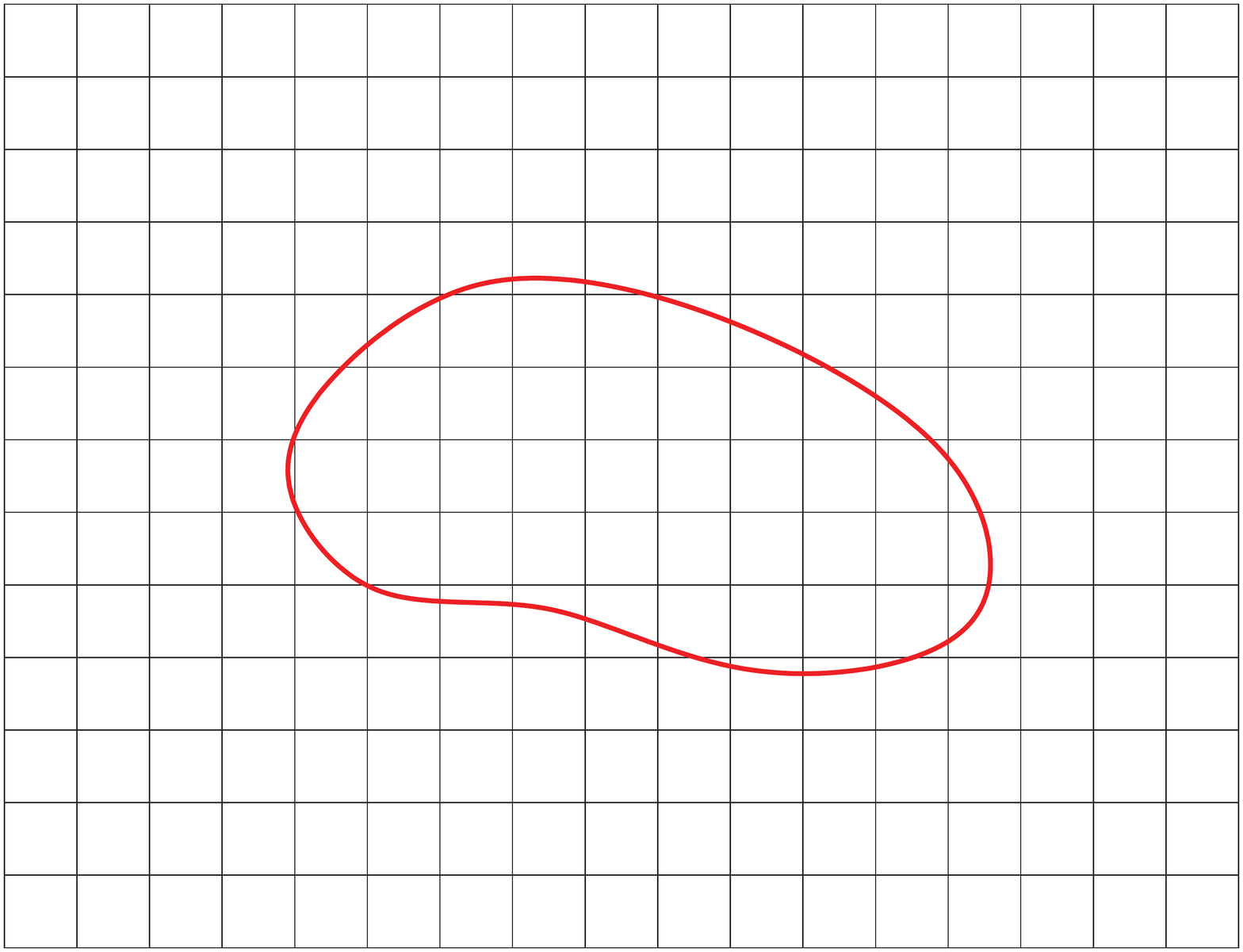}
\includegraphics[totalheight=1.2in]{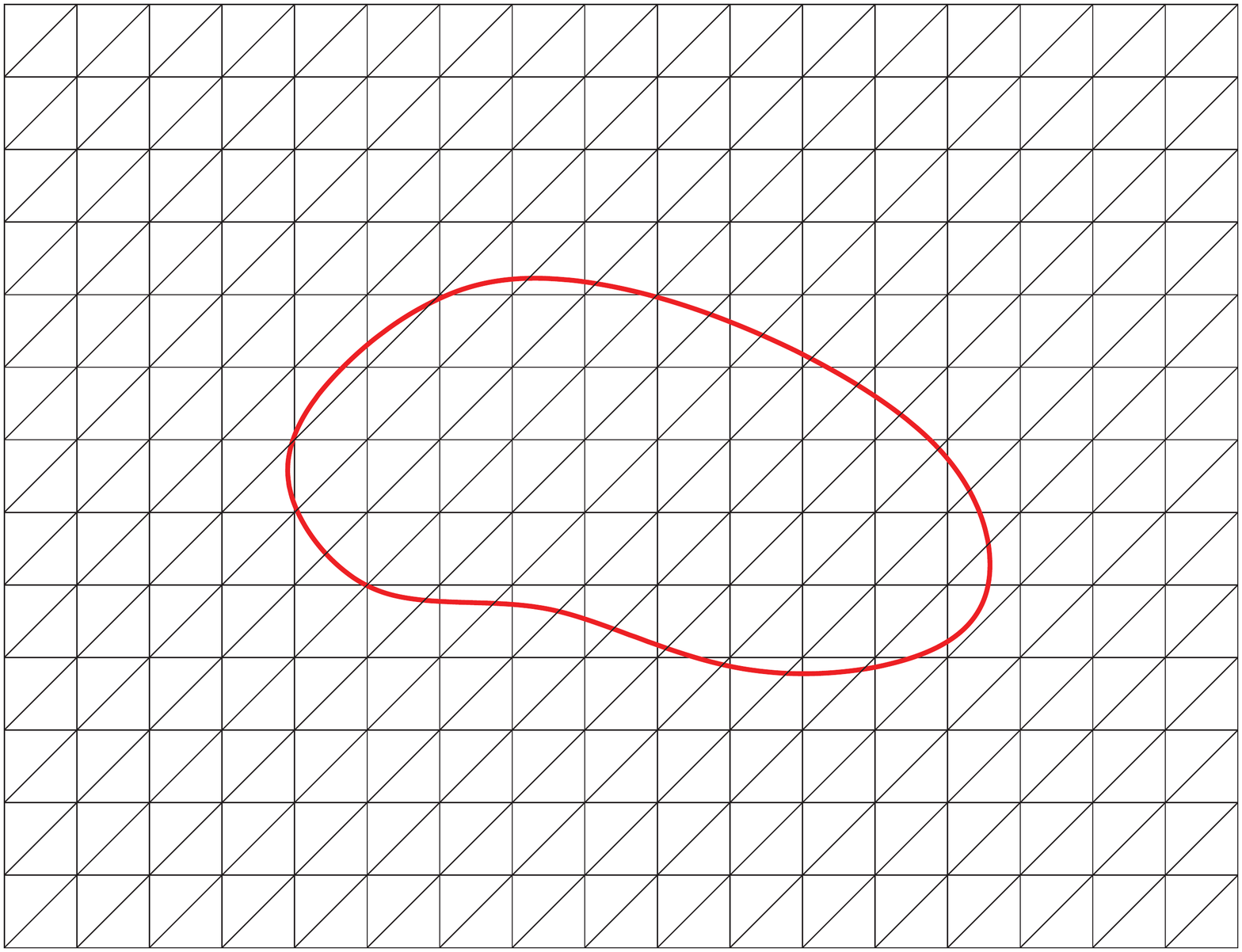}
\caption{\small from left: a simulation domain, a body-fitting triangular mesh, a non-body-fitting Cartesian mesh, and a non-body-fitting triangular mesh.}
\end{figure}

The discontinuous Galerkin (DG) methods for elliptic boundary value problems can be traced back to 1970s (see \cite{Babuska_Zlamal_Nonconform_FEM_Penalty, Reed_Hill_Triangular_Transport_EQ}) and they become increasingly popular recently as indicated by these survey articles and books \cite{Arnold_Brezzi_Cockburn_Marini_Unified_Analysis_DG, Cockburn_Karniadakis_Shu_DG_Book, Hesthaven_Warburton_Nodal_DG, Riviere_DG_book}. Because there is no continuity imposed on the approximating function across the element boundary, DG methods can locally perform $h$-, $p$-, and $hp$- refinement flexibly and efficiently. For elliptic and parabolic equations, the interior penalty DG (IPDG) methods \cite{Arnold_IPFE, Douglas_Dupont_Penalty_Elliptic_Parabolic, Riviere_Wheeler_Girault_DG, Wheeler_FE_IP} are well understood and widely used. The main feature of IPDG methods is that penalty terms are added on interior edges to stabilize the bilinear form of the scheme, so that the linear system is positive definite. In \cite{XHe_Thesis_Bilinear_IFE,XHe_TLin_YLin_DG}, the IFE and IPDG ideas were combined together for solving interface problems on Cartesian meshes with local refinement capability. To alleviate the issue of higher degrees of freedom in usual DG formulation, authors in \cite{He_Lin_Lin_Selective_DG} considered the so-called {\it selective} DG-IFE methods that employ DG formulation in selected elements while using the usual Galerkin formulation in the rest of the solution domain. Numerical examples have demonstrated that these DG-IFE methods can converge optimally, and our goal in this article is to theoretically establish the optimal {\it a priori} error estimates for DG-IFE methods that were discussed in
\cite{XHe_Thesis_Bilinear_IFE,XHe_TLin_YLin_DG,He_Lin_Lin_Selective_DG}.

The rest of the paper is organized as follows. In Section 2, we recall the DG-IFE methods originally proposed in \cite{XHe_Thesis_Bilinear_IFE,XHe_TLin_YLin_DG}. In Section 3, we present {\it a priori} error estimates for these DG-IFE methods. An error estimate in a mesh-dependant energy norm is derived, and this
error estimate is optimal according to the polynomials used in the IFE spaces. In Section 4, numerical experiments are provided to demonstrate features of DG-IFE methods. Brief conclusions are given in Section 5.

\section{Discontinuous Galerkin Immersed Finite Element Methods}
In this paper, we adopt notations and norms of usual Sobolev spaces.
For $r> 1$ and any subset $G\subseteq \Omega$ that is cut through by $\Gamma$,
we use the following function spaces:
\begin{equation*}
  \tilde H^{r}(G)=\{v\in H^1(G):\ v|_{G\cap\Omega^s}\in
H^{r}(G\cap\Omega^s), s=+\ \text{or}\ -\},~~~~~
  \tilde H^{r}_0(G) = \tilde H^{r}(G)\cap H^1_0(G),
\end{equation*}
equipped with the norm
\begin{equation*}
\|v\|^2_{\tilde H^{r}(G)}=\|v\|^2_{H^{r}(G\cap\Omega^-)}+\|v\|^2_{H^{r}(G\cap\Omega^+)},\ \forall v\in
\tilde H^{r}(G).
\end{equation*}
From now on, we use $C$ with or without subscripts to denote generic positive constants, possibly different at different occurrences, but they are independent of the mesh size and interface.

Let $\{\mathcal{T}_h\}$ with $0<h<1$ be a family of triangular or rectangular Cartesian meshes of $\Omega$.
An element cut through by the interface is called an interface element; otherwise, it is called a non-interface element.
For each mesh $\mathcal{T}_h$, we let $\mathcal{T}_h^i$ be the set of interface elements of $\mathcal {T}_h$ and
$\mathcal{T}_h^n$ be the set of non-interface elements. We denote by  $\mathcal{E}_h$ the set of edges of $\mathcal {T}_h$. Also, let $\mathcal {\mathring{E}}_h$ and $\mathcal {E}_h^b$ be the set of interior edges and boundary edges of $\mathcal {T}_h$, respectively.
Similarly, if an edge is cut through by the interface, it is called an interface edge; otherwise, it is called a non-interface edge. Let $\mathcal {E}_h^i$ and  $\mathcal {{E}}_h^n$ be the set of interface edges and non-interface edges, respectively. Moreover, we use $\mathcal {\mathring{E}}_h^i$ and  $\mathcal {\mathring{E}}_h^n$ to denote the set of interior interface edges and interior non-interface edges, respectively. Without loss of generality, we assume that elements in $\mathcal {T}_h$
satisfy the following conditions:
\begin{description}
  \item[(H1)] If one edge of an element meets $\Gamma$ at more than one point, then this edge is part of $\Gamma$.
  \item[(H2)] If $\Gamma$ meets the boundary of an element at two points, then these two points must be on different edges of this element.
\end{description}

For every interface element $K\in \mathcal{T}_h^i$, we assume its boundary intersects with the interface $\Gamma$ at points $D$ and $E$.
Then, the line segment $\overline{DE}$ divides $K$ into two sub-elements $K^+$ and $K^-$
with $K=K^+\cup K^-\cup \overline{DE}$. With a given mesh $\mathcal{T}_h$ on $\Omega$, we define the following broken Sobolev spaces:
\begin{align*}
  \tilde H^2(\mathcal{T}_h)=\{v\in L^2(\Omega): &~ \forall
K\in\mathcal{T}_h^n, v|_K\in H^2(K);\\
   &~ \forall
K\in\mathcal{T}_h^i, v|_{K}\in H^1(K),\ \ v|_{K^s}\in H^2(K^s), s=+,-\}.
\end{align*}
and
\begin{equation*}
  \tilde H_0^2(\mathcal{T}_h)=\{v\in \tilde H^2(\mathcal{T}_h):
v|_{\partial\Omega}=0\}.
\end{equation*}

We now recall some standard notations for describing IPDG methods
\cite{ZChen_FEM, Hesthaven_Warburton_Nodal_DG, Riviere_DG_book}. For each edge $B$, we associate a unit normal vector $\mathbf{n}_B$.
If $B\in\mathcal {\mathring{E}}_h$, we let $K_{B,1}$ and $K_{B,2}$ be two elements
that share $B$ as the common edge and let $\mathbf{n}_B$ be
the outward normal with respect to $K_{B,1}$.
If $B\in \mathcal {E}_h^b$, $\mathbf{n}_B$ is taken to be the unit outward vector normal to
$\partial\Omega$. For a function $u$ defined on $K_{B,1}\cup K_{B,2}$, we denote its average
and jump over $B\in\mathcal {\mathring{E}}_h$ by
\begin{equation*}
\aver{u}_B=\frac{1}{2}((u|_{K_{B,1}})|_B+(u|_{K_{B,2}})|_B),\ \
\jump{u}_B=(u|_{K_{B,1}})|_B-(u|_{K_{B,2}})|_B.
\end{equation*}
If $B$ is a boundary edge, we set
\begin{equation*}
\aver{u}_B=\jump{u}_B=u|_B.
\end{equation*}
For simplicity, we usually drop the subscript $B$ from these notations if there is no danger to
cause any confusion.
\par
To obtain a variational form for the interface problem
\eqref{eq: pde} - \eqref{eq: jump2}, we multiply \eqref{eq: pde} by a test function $v\in
\tilde H_0^2(\mathcal{T}_h)$, integrate both sides on each element
$K\in\mathcal{T}_h$, and apply the Green's formula to have
\begin{equation}\label{eq: Green's formula}
\int_{K}\beta\nabla u\cdot\nabla vdX-\int_{\partial K}\beta\nabla
u\cdot{\mathbf{n}_{K}}vds=\int_KfvdX.
\end{equation}
Note that \eqref{eq: Green's formula} holds regardless whether $K$ is a non-interface element or an interface element. For $K \in \mathcal{T}_h^n$, the derivation follows from the standard procedure. When $K$ is an interface element, \eqref{eq: Green's formula} follows from applying the Green's formula piecewisely over sub-elements of $K$
determined according to the smoothness of $u$ and $v$, then summing up over $K$ and applying the flux continuity \eqref{eq: jump2}. Summarizing (2.1) over all elements we obtain
\begin{equation}\label{eq: weak form 1}
\sum\limits_{K\in\mathcal{T}_h}\int_K\beta\nabla u\cdot \nabla vdX
-\sum\limits_{B\in \mathcal{\mathring E}_h}\int_{B}\aver{\beta\nabla u\cdot \mathbf{n}_B}\jump{v}ds
=\int_{\Omega}fvdX.
\end{equation}
Since the solution $u$ is continuous almost everywhere in $\Omega$, we can assume
\begin{equation}\label{eq: penalty term}
\epsilon\sum\limits_{B\in \mathcal{\mathring E}_h}\int_{B}\aver{\beta\nabla v\cdot \mathbf{n}_B}\jump{u}ds=0,~~~~~
\sum\limits_{B\in \mathcal{\mathring E}_h}\frac{\sigma^0_B}{|B|^{\alpha}}\int_{B}\jump{u}\jump{v}ds=0,
\end{equation}
for any constants $\epsilon$, $\alpha>0$, and $\sigma^0_B\geq 0$. Here $|B|$ denotes the length of
$B$. Adding the two terms in \eqref{eq: penalty term} to \eqref{eq: weak form 1}, we obtain the weak form of interface problem
\eqref{eq: pde} - \eqref{eq: jump2}: Find $u\in \tilde H^2_0(\Omega)$ such that
\begin{equation}\label{eq: weak form}
  a_\epsilon(u,v) = (f,v), ~~~  \forall v\in \tilde H^2_0(\mathcal{T}_h),
\end{equation}
where the bilinear form $a_\epsilon(\cdot,\cdot)$: $H_h(\Omega)\times H_h(\Omega)\rightarrow \mathbb{R}$ is
\begin{eqnarray}
a_\epsilon(w,v)
&=&
\sum\limits_{K\in\mathcal{T}_h}\int_K\beta\nabla w\cdot \nabla v dX
-\sum\limits_{B\in \mathcal{\mathring E}_h}\int_{B}\aver{\beta\nabla w\cdot \mathbf{n}_B}\jump{v}ds\nonumber\\
& &
+\epsilon\sum\limits_{B\in \mathcal{\mathring E}_h}\int_{B}\aver{\beta\nabla v\cdot \mathbf{n}_B}\jump{w}ds
+\sum\limits_{B\in \mathcal{\mathring E}_h}\frac{\sigma^0_B}{|B|^{\alpha}}\int_{B}\jump{w}\jump{v}ds,
\label{eq: DG bilinear form}
\end{eqnarray}
and $H_h(\Omega)=\tilde H_0^2(\Omega)+\tilde H_0^2(\mathcal{T}_h)$. The weak form derived here for the interface problem \eqref{eq: pde}-\eqref{eq: jump2}
is in the same format as the standard weak form used in DG finite element methods for the usual elliptic boundary value problems
\cite{ZChen_FEM, Hesthaven_Warburton_Nodal_DG, Riviere_DG_book}. As suggested by DG finite element methods, the parameter $\epsilon$ is usually chosen as $-1$, $0$, or $1$. Note that the bilinear form $a_\epsilon(\cdot,\cdot)$ is symmetric if $\epsilon=-1$ and is nonsymmetric otherwise.

We now introduce the IFE approximation of the broken space $\tilde H_0^2(\mathcal{T}_h)$.
For every element $K\in\mathcal{T}_h$, denote by $A_i$, $i=1,\cdots,d_K$, the vertices of $K$. Here $d_K = 3$
or $d_K = 4$ depending on whether $K$ is a triangular or rectangular element.
On each non-interface element $K\in\mathcal{T}_h^n$, we let $\psi_i$, $i = 1,\cdots, d_K$ be the standard linear or bilinear finite element nodal basis associated with the vertex $A_i$ of $K$. The local FE space on $K \in\mathcal{T}_h^n$ is the defined as
\begin{equation*}
S_h(K)=span\{\psi_i: 1\leq i\leq d_K\}.
\end{equation*}
On an interface element $K\in \mathcal{T}_h^i$, we let $\phi_i$, $i = 1,\cdots, d_K$ be the linear \cite{ZLi_TLin_YLin_RRogers_linear_IFE, ZLi_TLin_XWu_Linear_IFE} or bilinear \cite{XHe_TLin_YLin_Bilinear_Approximation,TLin_YLin_RRogers_MRyan_Rectangle} IFE nodal basis associated with vertex $A_i$. We let local IFE space on
$K\in \mathcal{T}_h^i$ be
\begin{equation*}
S_h(K)=span\{\phi_i: 1\leq i\leq d_K\}.
\end{equation*}
Then, we define the discontinuous IFE space over the
mesh $\mathcal{T}_h$ as follows:
\begin{equation*}
 S_{h}(\mathcal{T}_h)=\{v\in L^2(\Omega): v|_K\in S_h(K)\},~~~~
 \mathring{S}_{h}(\mathcal{T}_h) = \{v\in S_{h}(\mathcal{T}_h): v|_{\partial \Omega} = 0\}.
\end{equation*}
One can easily see that $\mathring{S}_{h}(\mathcal{T}_h)$ is a subspace of $H_h(\Omega)$.

Finally, we state the DG-IFE methods for the interface problem \eqref{eq: pde} - \eqref{eq: jump2} as: Find $u_h\in \mathring{S}_{h}(\mathcal{T}_h)$ such that
\begin{equation}\label{eq: DG method}
a_\epsilon(u_h,v_h)=(f,v_h),\ \ \forall v_h\in \mathring{S}_{h}(\mathcal{T}_h).
\end{equation}

\section{{\it A Priori} Error Estimation}
In this section, we derive {\it a priori} error estimates for the DG-IFE methods \eqref{eq: DG method} in an energy norm $\|\cdot\|_h: H_h(\Omega) \rightarrow \mathbb{R}$ defined as follows
\begin{equation}\label{eq: energy norm}
\|v\|_h=
\left(\sum\limits_{K\in\mathcal{T}_h}\int_K\beta\nabla v\cdot \nabla vdX
+\sum\limits_{B\in \mathcal{\mathring E}_h}\frac{\sigma_B^0}{|B|^{\alpha}}\int_{B}\jump{v}\jump{v}ds\right)^{1/2}.
\end{equation}
We first present a few lemmas required in the error analysis. By the standard
scaling argument, one can show the following trace inequalities \cite{Riviere_DG_book}:

\begin{lemma}
(Standard trace inequalities)~ Let $K$ be a triangle or rectangle with diameter $h_K$, and $B$ be an edge of $K$. There exists a constant $C$ such that
\begin{eqnarray}
\|v\|_{L^2(B)}&\leq&
C|B|^{1/2}|K|^{-1/2}(\|v\|_{L^2(K)}+h_K\|\nabla
v\|_{L^2(K)}),~~~\forall v\in H^1(K),  \label{eq: trace 1}\\
\|\nabla v\|_{L^2(B)}&\leq&
C|B|^{1/2}|K|^{-1/2}(\|\nabla
v\|_{L^2(K)}+h_K\|\nabla^2 v\|_{L^2(K)}),~~~\forall v\in H^2(K),
\label{eq: trace 2}
\end{eqnarray}
where $\nabla^2 v$ is the Hessian of $v$.
\end{lemma}
On an interface element $K\in \mathcal{T}_h^i$, we recall from \cite{ZLi_TLin_XWu_Linear_IFE, TLin_YLin_RRogers_MRyan_Rectangle} that the local IFE space $S_h(K)\subset H^1(K)$. This implies that the trace inequality \eqref{eq: trace 1} is valid for
$v\in S_h(K)$. However, since $S_h(K) \not \subset H^2(K)$ for $K \in \mathcal{T}_h^i$ in general, the second inequality \eqref{eq: trace 2} cannot be applied to functions
in $S_h(K)$. Nevertheless, in \cite{TLin_YLin_XZhang_PPIFE_Elliptic, XZhang_PHDThesis}, this trace inequality has been extended to IFE functions. We recall this result in
the following lemma.

\begin{lemma}
(Trace inequalities for IFE functions)~ Let $\mathcal{T}_h$ be a Cartesian triangular or rectangular mesh and let $K\in \mathcal{T}_h$ be an interface triangle or rectangle with diameter $h_K$ and let $B$ be an edge of $K$. There exists a constant $C$, independent of interface location but depending on the jump of the coefficient $\beta$, such that for every linear or bilinear IFE function $v$ defined on $K$, the following inequality hold
\begin{equation}\label{eq: trace IFE}
  \|\beta \nabla v\cdot \mathbf{n}_B\|_{L^2(B)}\leq
Ch_K^{-1/2}\|\sqrt{\beta}\nabla v\|_{L^2(K)}.
\end{equation}
\end{lemma}

We now describe the interpolation with IFE functions. For $K\in \mathcal{T}_h^n$, the local interpolation operator is defined as $I_{h,K}^n: C(K)\to S_h(K)$:
\begin{equation*}
  I_{h,K}^n u(X) = \sum_{i=1}^{d_K}u(A_i)\psi_i(X),~~~K\in \mathcal{T}_h^n.
\end{equation*}
For $K\in \mathcal{T}_h^i$, the local interpolation operator is defined as $I_{h,K}^i: C(K)\to S_h(K)$:
\begin{equation*}
  I_{h,K}^i u(X) = \sum_{i=1}^{d_K}u(A_i)\phi_i(X),~~~K\in \mathcal{T}_h^i.
\end{equation*}
On each non-interface element, we have the standard approximation theory for the finite element interpolation:
\begin{equation}\label{eq: interpolation error}
\|u-I^n_{h,K}u\|_{L^2(K)} + h_K|u-I^n_{h,K}u|_{H^1(K)} \leq Ch_K^{2}|u|_{H^2(K)},~~~~~\forall K\in \mathcal{T}_h^n.
\end{equation}
On each interface element, the approximation property of the IFE interpolation proved in \cite{XHe_TLin_YLin_Bilinear_Approximation, ZLi_TLin_XWu_Linear_IFE} provides similar error bounds as follows:
\begin{equation}\label{eq: IFE interpolation error}
\|u-I_{h,K}^i u\|_{L^2(K)}+h_K|u-I_{h,K}^iu|_{H^1(K)}\leq
Ch_K^2\|u\|_{\tilde H^2(K)},~~~~~\forall K\in \mathcal{T}_h^i,
\end{equation}
where the constant $C$ is independent of interface location. For $u\in \tilde H^2(\Omega)$, let $I_h: \tilde H^2(\Omega) \rightarrow S_{h}(\mathcal{T}_h)$ be the interpolation defined by
\begin{equation}\label{eq: interpolation}
  (I_hu)|_{K}=\left\{\begin{array}{l}
I^n_{h,K}u,\ \ K\in \mathcal{T}_h^n,\vspace{1mm}\\
I^i_{h,K}u,\ \ K\in \mathcal{T}_h^i.
\end{array}
\right.
\end{equation}
Multiplying $h_K^{-1}$ on both sides of \eqref{eq: interpolation error} and \eqref{eq: IFE interpolation error}, then summing up for all non-interface and interface elements, we can obtain an interpolation error bound on the domain $\Omega$ as stated in the next lemma.
\begin{lemma}
For $u\in \tilde H^2(\Omega)$, satisfying the interface jump conditions \eqref{eq: jump1} and \eqref{eq: jump2}, there exists a constant $C$ such that
\begin{equation}\label{eq: interpolation global}
\Big(\sum\limits_{K\in\mathcal{T}_h}h_K^{-2}\|u-I_hu\|_{L^2(K)}^2\Big)^{1/2}
+\Big(\sum\limits_{K\in\mathcal{T}_h}|u-I_hu|_{H^1(K)}^2\Big)^{1/2}
\leq Ch\|u\|_{\tilde H^2(\Omega)},
\end{equation}
where $h=\max\limits_{K\in \mathcal{T}_h} h_K$.
\end{lemma}
The following lemma provides the approximation property of $I_hu$ in the energy norm $\|\cdot\|_h$.
\begin{lemma}
Assume $\alpha\leq 1$ in the energy norm \eqref{eq: energy norm}. For every $u\in \tilde H^2(\Omega)$, satisfying the interface jump conditions \eqref{eq: jump1} and \eqref{eq: jump2}, there exists a constant $C$ such that
\begin{equation}\label{eq: interpolation h norm}
\|u-I_hu\|_h\leq Ch\|u\|_{\tilde H^2(\Omega)}.
\end{equation}
\end{lemma}
\begin{proof}
By the definition of $\|\cdot\|_h$, we have
\begin{equation}\label{eq: energy norm tmp1}
  \|u-I_hu\|_h^2 = \sum\limits_{K\in\mathcal{T}_h}\int_K\beta|\nabla
(u-I_hu)|^2dX+\sum\limits_{B\in \mathcal{\mathring
E}_h}\frac{\sigma_B^0}{|B|^{\alpha}}\|\jump{u-I_hu}\|_{L^{2}(B)}^2.
\end{equation}
For the first term on the right hand side, we use the estimate \eqref{eq: interpolation global} to have
\begin{equation}\label{eq: interpolation error temp1}
\sum\limits_{K\in\mathcal{T}_h}\int_K\beta|\nabla
(u-I_hu)|^2dX\leq \beta_{max}\sum\limits_{K\in\mathcal{T}_h}\|\nabla
(u-I_hu)\|_{L^2(K)}^2\leq \beta_{max}h^2\|u\|^2_{\tilde H^2(\Omega)}
 \end{equation}
where $\beta_{max} = \max\{\beta^-,\beta^+\}$. Now we bound the second term in \eqref{eq: energy norm tmp1}. Using the standard trace equality \eqref{eq: trace 1} and the approximation properties \eqref{eq: interpolation error} or \eqref{eq: IFE interpolation error}, we have
 \begin{eqnarray*}
\frac{\sigma_B^0}{|B|^{\alpha}}\|\jump{u-I_hu}\|^2_{L^2(B)}&\leq &
\frac{\sigma_B^0}{|B|^{\alpha}}(\|(u-I_hu)|_{K_{B,1}}\|^2_{L^2(B)}+\|(u-I_hu)|_{K_{B,2}}\|^2_{L^2(B)})\\
&\leq&
Ch_{K_{B,1}}^{-1-\alpha}\left(\|u-I_hu\|_{L^2(K_{B,1})}^2+h_{K_{B,1}}^2\|\nabla(u-I_hu)\|^2_{L^2(K_{B,1})}\right)\\
&&+Ch_{K_{B,2}}^{-1-\alpha}\left(\|u-I_hu\|^2_{L^2(K_{B,2})}+h_{K_{B,2}}^2\|\nabla(u-I_hu)\|^2_{L^2(K_{B,2})}\right)\\
&\leq& Ch_{K_{B,1}}^{3-\alpha}\|u\|^2_{V(K_{B,1})} + Ch_{K_{B,2}}^{3-\alpha}\|u\|^2_{V(K_{B,2})}\\
&\leq& Ch^{3-\alpha}\left(\|u\|^2_{V(K_{B,1})}+\|u\|^2_{V(K_{B,2})}\right),
 \end{eqnarray*}
where $V(K)=H^2(K)$ for $K\in\mathcal{T}_h^n$ and $V(K)=\tilde H^2(K)$ for $K\in\mathcal{T}_h^i$, and $h=\max\limits_{K\in \mathcal{T}_h} h_K$. Also, the second inequality is due to the shape-regular property of Cartesian triangular or rectangular meshes $h_{K_B,i}\leq C|B|$, $i=1,2$. Thus, for $\alpha\leq 1$, we get
\begin{equation}\label{eq: interpolation error temp2}
 \sum\limits_{B\in \mathcal{\mathring E}_h}\frac{\sigma_B^0}{|B|^{\alpha}}\|\jump{u-I_hu}\|^2_{L^2(B)}\leq Ch^2\|u\|^2_{\tilde H^2(\Omega)}.
 \end{equation}
 Finally, combining \eqref{eq: interpolation error temp1} and \eqref{eq: interpolation error temp2}, we get \eqref{eq: interpolation h norm}.
\end{proof}

The coercivity of the bilinear form $a_\epsilon(\cdot,\cdot)$ is analyzed in the following lemma.
\begin{lemma}
There exists a constant $\kappa>0$ such that
\begin{equation}\label{eq: coercivity}
 a_\epsilon(v_h,v_h)\geq\kappa\|v_h\|_h^2,\ \ \forall v_h\in
 \mathring{S}_{h}(\mathcal{T}_h)
\end{equation}
holds for $\epsilon=1$ unconditionally and holds for $\epsilon=0 \text{ or }-1$ under the conditions that the penalty parameter $\sigma_B^0$ is large enough and $\alpha\geq 1$.
\end{lemma}

\begin{proof}
From the definition of $a_\epsilon(\cdot,\cdot)$, we have
\begin{equation}\label{eq: coercivity temp1}
  a_\epsilon(v_h,v_h) =
  \sum\limits_{K\in\mathcal{T}_h}\int_K\beta|\nabla v_h|^2 dX
+(\epsilon-1)\sum\limits_{B\in \mathcal{\mathring E}_h}\int_{B}\aver{\beta\nabla v_h\cdot
\mathbf{n}_B}\jump{v_h}ds+\sum\limits_{B\in \mathcal{\mathring E}_h}\frac{\sigma^0_B}{|B|^{\alpha}}\int_{B}\jump{v_h}^2ds.
\end{equation}
We first note that, when $\epsilon=1$, the coercivity follows directly from \eqref{eq: coercivity temp1} and the definition of $\|\cdot\|_h$. If $\epsilon=0 \text{ or }-1$, we need to bound the second term on the right hand side of \eqref{eq: coercivity temp1}. For each $B\in \mathcal{\mathring E}_h$, recall that $K_{B,i}\in \mathcal{T}_h$, $i=1,2$  are two elements sharing $B$ as their common edge.
If $K_{B,i}$, $i=1$ or 2 is a non-interface element, by the
trace inequality \eqref{eq: trace 2} and inverse inequalities, we have
\begin{eqnarray}
  \|(\beta\nabla v_h\cdot\mathbf{n}_B)|_{K_{B,i}}\|_{L^2(B)}
  &\leq& \beta_{\max}\|(\nabla v_{h})|_{K_{B,i}}\|_{L^2(B)} \nonumber\\
  &\leq&  C\beta_{\max}h_{K_{B,i}}^{-\frac{1}{2}}\|\nabla v_{h}\|_{L^2(K_{B,i})}\nonumber\\
  &\leq& C\frac{\beta_{\max}}{\sqrt{\beta_{\min}}}h_{K_{B,i}}^{-\frac{1}{2}}\|\sqrt{\beta}\nabla v_h\|_{L^2(K_{B,i})}, \label{eq: non-interface edge}
\end{eqnarray}
where $\beta_{\min} = \min\{\beta^-,\beta^+\}$, and $\beta_{\max} = \max\{\beta^-,\beta^+\}$.
Then, using the assumption that $\alpha \geq 1$ and by either the estimate \eqref{eq: non-interface edge} or IFE trace inequality \eqref{eq: trace IFE} depending on whether
the element is a non-interface element or an interface element, we have
\begin{eqnarray*}
\int_B\aver{\beta\nabla v_h\cdot \mathbf{n}_B}\jump{v_h}ds
&\leq& \|\aver{\beta\nabla v_h\cdot \mathbf{n}_B}\|_{L^2(B)}\|\jump{v_h}\|_{L^2(B)}\\
&\leq& \frac{1}{2}\left(\|(\beta\nabla v_h\cdot \mathbf{n}_B)|_{K_{B,1}}\|_{L^2(B)}
+\|(\beta\nabla v_h\cdot\mathbf{n}_B)|_{K_{B,2}}\|_{L^2(B)}\right)\|\jump{v_h}\|_{L^2(B)}\\
&\leq& \frac{C}{2}\left(h_{K_{B,1}}^{-\frac{1}{2}}\|\sqrt{\beta}\nabla v_h\|_{L^2(K_{B,1})}
+h_{K_{B,2}}^{-\frac{1}{2}}\|\sqrt{\beta}\nabla v_h\|_{L^2(K_{B,2})}\right)\|\jump{v_h}\|_{L^2(B)}\\
&\leq& C\left(\|\sqrt{\beta}\nabla v_h\|_{L^2(K_{B,1})}^2
+\|\sqrt{\beta}\nabla v_h\|_{L^2(K_{B,2})}^2\right)^{\frac{1}{2}}\frac{1}{|B|^{\alpha/2}}\|\jump{v_h}\|_{L^2(B)}.
\end{eqnarray*}
Summing over all interior edges and using the Young's inequality,  we have
\begin{eqnarray}
&&\sum\limits_{B\in \mathcal{\mathring E}_h}\int_B\aver{\beta\nabla v_h\cdot
\mathbf{n}_B}\jump{v_h}ds\nonumber\\
&\leq&
C\sum\limits_{B\in \mathcal{\mathring E}_h}
\left(\|\sqrt{\beta}\nabla v_h\|_{L^2(K_{B,1})}^2+\|\sqrt{\beta}\nabla v_h\|_{L^2(K_{B,2})}^2\right)^{1/2}
\frac{1}{|B|^{\alpha/2}}\|\jump{v_h}\|_{L^2(B)}\nonumber\\
&\leq&
C\left(\sum\limits_{B\in \mathcal{\mathring E}_h}\frac{1}{|B|^{\alpha}}\|\jump{v_h}\|^2_{L^2(B)}\right)^{1/2}\left(\sum\limits_{B\in
\mathcal{\mathring E}_h}\left(\|\sqrt{\beta}\nabla
v_h\|_{L^2(K_{B,1})}^2+\|\sqrt{\beta}\nabla
v_h\|_{L^2(K_{B,2})}^2\right)\right)^{1/2}\nonumber\\
&\leq&
\frac{\delta}{2}\sum\limits_{K\in
\mathcal{T}_h}\|\sqrt{\beta}\nabla v_h\|_{L^2(K)}^2+\frac{C}{2\delta}\sum\limits_{B\in \mathcal{\mathring E}_h}\frac{1}{|B|^{\alpha}}\|\jump{v_h}\|^2_{L^2(B)}.\label{eq: coercivity temp2}
 \end{eqnarray}
Then, for $\epsilon=0$, we can choose
\begin{equation*}
\delta=1\ \  \text{and}\ \
 \sigma_B^0> \frac{C}{2},
\end{equation*}
and for $\epsilon=-1$, we can choose
\begin{equation*}
 \delta=\frac{1}{2}\ \  \text{and}\ \
 \sigma_B^0>2C.
\end{equation*}
 Substituting these parameters in \eqref{eq: coercivity temp2} and then putting it in \eqref{eq: coercivity temp1}, we obtain \eqref{eq: coercivity}.
\end{proof}

We also need an error bound for the IFE interpolation $I_hu$ on interface edges which has been proved in \cite{TLin_YLin_XZhang_PPIFE_Elliptic}. We present the result in the following lemma.
\begin{lemma}
For every $u\in \tilde H^3(\Omega)$, satisfying the interface jump conditions \eqref{eq: jump1} and \eqref{eq: jump2}, there exists a constant $C$ independent of interface location such that
\begin{equation}\label{eq: interpolation error edge}
\|\beta(\nabla(u-I_hu))|_K\cdot \mathbf{n}_B\|^2_{L^2(B)}\leq
C\left(h_K^2\|u\|_{\tilde H^3(\Omega)}^2+h_K\|u\|_{\tilde H^2(K)}^2\right),
\end{equation}
where $K$ is an interface element and $B$ is one of its interface
edge.
\end{lemma}

The assumptions of $\alpha$ in Lemma 3.4 and Lemma 3.5 suggest that we should choose $\alpha = 1$ in our DG formulation \eqref{eq: DG method}. Now we are ready to prove the \textit{a priori} error estimate for DG-IFE method \eqref{eq: DG method}.
\begin{theorem} \label{th:DG_convergence}
Let $u\in \tilde H^3(\Omega)$ be the exact solution to the interface problem \eqref{eq: pde} to \eqref{eq: jump2} and $u_h\in S_{h}(\mathcal{T}_h)$ be the solution to \eqref{eq: DG method} with $\alpha=1$, $\epsilon = -1$, $0$, or $1$. Then there exists a constant $C$ such that
\begin{equation}\label{eq: DG error estimates}
  \|u-u_h\|_h\leq Ch\|u\|_{\tilde H^3(\Omega)}.
\end{equation}
\end{theorem}
\begin{proof}
Subtracting the weak form \eqref{eq: weak form} from the DG-IFE scheme \eqref{eq: DG method}, we get
\begin{equation}\label{eq: orthogonality}
  a_\epsilon(u-u_h,v_h)=0,\ \ \forall v_h\in \mathring{S}_{h}(\mathcal{T}_h).
\end{equation}
For every $w_h\in \mathring{S}_{h}(\mathcal{T}_h)$, using \eqref{eq: orthogonality} and the coercivity \eqref{eq: coercivity}, we have
\begin{eqnarray}\nonumber
 \kappa\|u_h-w_h\|_h^2
 &\leq& a_\epsilon(u_h-w_h,u_h-w_h)=a_\epsilon(u-w_h,u_h-w_h)\\
 \nonumber
 &\leq& \left|\sum\limits_{K\in\mathcal{T}_h}\int_K\beta\nabla (u-w_h)\cdot \nabla(u_h-w_h)dX\right|
 +\left|\sum\limits_{B\in \mathcal{\mathring E}_h}\int_{B}\aver{\beta\nabla (u-w_h)\cdot \mathbf{n}_B}\jump{u_h-w_h}ds\right|\\ \nonumber
 &&+\left|\sum\limits_{B\in \mathcal{\mathring E}_h}\int_{B}\aver{\beta\nabla (u_h-w_h)\cdot \mathbf{n}_B}\jump{u-w_h}ds\right|
 +\left|\sum\limits_{B\in \mathcal{\mathring E}_h}\frac{\sigma^0_B}{ |B|^{\alpha}}\int_{B}\jump{u-w_h}\jump{u_h-w_h}ds\right|\\
 &\triangleq& T_1+T_2+T_3+T_4. \label{eq: DG error temp1}
\end{eqnarray}
We proceed to bound the terms $T_i, i = 1, 2, 3, 4$ in \eqref{eq: DG error temp1}. By the Cauchy-Schwarz inequality and Young's inequality with parameter $\delta > 0$, we can easily bound $T_1$ and $T_2$:
\begin{eqnarray}\nonumber
T_1&\leq & \left(\sum\limits_{K\in\mathcal{T}_h}\|\sqrt{\beta}\nabla
(u-w_h)\|^2_{L^2(K)}\right)^{1/2}\left(\sum\limits_{K\in\mathcal{T}_h}\|\sqrt{\beta}\nabla
(u_h-w_h)\|^2_{L^2(K)}\right)^{1/2}\\ \nonumber &\leq &\frac{1}{4\delta}\beta_{\max}\|\nabla
(u-w_h)\|_{L^2(\Omega)}^2+\delta\sum\limits_{K\in\mathcal{T}_h}\|\sqrt{\beta}\nabla
(u_h-w_h)\|^2_{L^2(K)}\\
&\leq& C(\delta)\|\nabla (u-w_h)\|_{L^2(\Omega)}^2+\delta\|u_h-w_h\|^2_h, \label{eq: DG error temp2}
 \end{eqnarray}
 and
\begin{eqnarray}\nonumber
T_2&\leq & C(\delta)\sum\limits_{B\in\mathcal{\mathring
E}_h}\frac{|B|^{\alpha}}{\sigma_B^0}\|\aver{\beta\nabla
(u-w_h)\cdot\mathbf{n}_B}\|^2_{L^2(B)}+\delta\sum\limits_{B\in\mathcal{\mathring E}_h}\frac{\sigma_B^0}{|B|^{\alpha}}\|\jump{u_h-w_h}\|^2_{L^2(B)}
\\
&\leq &C(\delta)\sum\limits_{B\in\mathcal{\mathring
E}_h}\frac{|B|^{\alpha}}{\sigma_B^0}\|\aver{\beta\nabla
(u-w_h)\cdot\mathbf{n}_B}\|^2_{L^2(B)}+\delta\|u_h-w_h\|^2_h, \label{eq: DG error temp3}
 \end{eqnarray}
where $C(\delta)$ emphasizes that this is a constant depending on $\delta$. For $T_3$, by the Cauchy-Schwarz inequality we have
\begin{equation}\label{eq: DG error temp4}
T_3\leq \sum\limits_{B\in \mathcal{\mathring E}_h}\|\aver{\beta\nabla (u_h-w_h)\cdot
\mathbf{n}_B}\|_{L^2(B)}\|\jump{u-w_h}\|_{L^2(B)}.
\end{equation}
First, using the standard trace equality \eqref{eq: trace 1}, we have
\begin{eqnarray*}
\|\jump{u-w_h}\|_{L^2(B)}&\leq &
\|(u-w_h)|_{K_{B,1}}\|_{L^2(B)}+\|(u-w_h)|_{K_{B,2}}\|_{L^2(B)}\\
&\leq&
Ch_{K_{B,1}}^{-1/2}\left(\|u-w_h\|_{L^2(K_{B,1})}+h_{K_{B,1}}\|\nabla(u-w_h)\|_{L^2(K_{B,1})}\right)\\
&&+Ch_{K_{B,2}}^{-1/2}\left(\|u-w_h\|_{L^2(K_{B,2})}+h_{K_{B,2}}\|\nabla(u-w_h)\|_{L^2(K_{B,2})}\right).
\end{eqnarray*}
Then, by the trace inequalities \eqref{eq: trace 2} or \eqref{eq: trace IFE} , we have
\begin{equation*}
\|\aver{\beta\nabla (u_h-w_h)\cdot \mathbf{n}_B}\|_{L^2(B)}\leq
C\left(h_{K_{B,1}}^{-1/2}\|\sqrt{\beta}\nabla(u_h-w_h)\|_{L^2(K_{B,1})}+h_{K_{B,2}}^{-1/2}\|\sqrt{\beta}\nabla(u_h-w_h)\|_{L^2(K_{B,2})}\right).
\end{equation*}
Substituting the above two bounds into \eqref{eq: DG error temp4} and applying Young's inequality, we obtain
\begin{equation}\label{eq: DG error temp5}
T_3\leq
C(\delta)\left(\sum\limits_{K\in\mathcal{T}_h}h_K^{-2}\|u-w_h\|^2_{L^2(K)}+\sum\limits_{K\in\mathcal{T}_h}\|\nabla(u-w_h)\|^2_{L^2(K)}\right)+\delta\|u_h-w_h\|^2_h.
\end{equation}
For $T_4$, we use the assumption $\alpha = 1$, the Cauchy-Schwarz inequality, and Young's inequality to have
\begin{eqnarray}\label{eq: DG error temp6}
T_4
&\leq&\sum\limits_{B\in \mathcal{\mathring E}_h}\left(\frac{1}{4\delta}\frac{\sigma^0_B}{|B|}\int_{B}\jump{u-w_h}\jump{u-w_h}ds+
\delta\frac{\sigma^0_B}{|B|}\int_{B}\jump{u_h-w_h}\jump{u_h-w_h}ds\right).
\end{eqnarray}
Again, by trace inequality \eqref{eq: trace 1}, we have
\begin{eqnarray}\nonumber
\frac{\sigma^0_B}{|B|}\int_{B}\jump{u-w_h}^2ds \nonumber
&\leq&C\frac{\sigma^0_B}{|B|}\left(\|(u-w_h)|_{K_{B,1}}\|^2_{L^2(B)}+\|(u-w_h)|_{K_{B,2}}\|^2_{L^2(B)}\right)\\
\nonumber &\leq&
Ch_{K_{B,1}}^{-2}\left(\|u-w_h\|_{L^2(K_{B,1})}+h_{K_{B,1}}\|\nabla(u-w_h)\|_{L^2(K_{B,1})}\right)^2\\
&&+Ch_{K_{B,2}}^{-2}\left(\|u-w_h\|_{L^2(K_{B,2})}+h_{K_{B,2}}\|\nabla(u-w_h)\|_{L^2(K_{B,2})}\right)^2.
\label{eq: DG error temp7}
\end{eqnarray}
Using \eqref{eq: DG error temp7} in \eqref{eq: DG error temp6}, we have
\begin{eqnarray}\label{eq: DG error temp8}
T_4\leq C(\delta)\left(\sum\limits_{K\in\mathcal{T}_h}h_K^{-2}\|u-w_h\|^2_{L^2(K)}+
\sum\limits_{K\in\mathcal{T}_h}\|\nabla(u-w_h)\|^2_{L^2(K)}\right)
+\delta\|u_h-w_h\|^2_h.
\end{eqnarray}
Substituting \eqref{eq: DG error temp2}, \eqref{eq: DG error temp3}, \eqref{eq: DG error temp5} and \eqref{eq: DG error temp8} into \eqref{eq: DG error temp1} and choosing $\delta=\kappa/ 8$, we obtain
\begin{eqnarray}\nonumber
\|u_h-w_h\|_h^2&\leq&
C\sum\limits_{K\in\mathcal{T}_h}\|\nabla(u-w_h)\|^2_{L^2(K)}+C\sum\limits_{B\in\mathcal{\mathring
E}_h}\frac{|B|}{\sigma_B^0}\|\aver{\beta\nabla
(u-w_h)\cdot\mathbf{n}_B}\|^2_{L^2(B)}\\
&&+C\sum\limits_{K\in\mathcal{T}_h}h_K^{-2}\|u-w_h\|^2_{L^2(K)}\label{eq: DG error temp9}
\end{eqnarray}
Now, we let $w_h$ be the IFE interpolation $I_hu$ in \eqref{eq: DG error temp9} and use the optimal approximation
capability of linear or bilinear DG-IFE spaces \eqref{eq: interpolation global} to have
\begin{equation}\label{eq: DG error temp10}
\|u_h-I_hu\|_h^2\leq
Ch^2\|u\|^2_{\tilde H^2(\Omega)}+Ch\sum\limits_{B\in\mathcal{\mathring E}_h}\sum_{i=1,2}\|(\beta\nabla
(u-I_hu)\cdot\mathbf{n}_B)|_{K_{B,i}}\|^2_{L^2(B)}.
\end{equation}
We now bound the second term on the right hand side of \eqref{eq: DG error temp10}.  If
$K_{B,i}$, $i=1$ or $2$ is a non-interface element, we use the
trace inequality \eqref{eq: trace 2} to obtain
\begin{eqnarray}
  \|(\beta\nabla
(u-I_hu)\cdot\mathbf{n}_B)|_{K_{B,i}}\|^2_{L^2(B)}
 &\leq& C(h_{K_{B,i}}^{-1}\|\nabla
(u-I_hu)\|^2_{L^2(K_{B,i})}+h_{K_{B,i}}\|\nabla^2 u\|^2_{L^2(K_{B,i})}) \nonumber \\
   &\leq&  Ch_{K_{B,i}}\|u\|^2_{H^2(K_{B,i})}.\label{eq: bound on noninterface elems}
\end{eqnarray}
If $K_{B,i}$ is an interface element, we use \eqref{eq: interpolation error edge} to get
\begin{eqnarray}
\|(\beta\nabla
(u-I_hu)\cdot\mathbf{n}_B)|_{K_{B,i}}\|^2_{L^2(B)}\leq
C\left(h_{K_{B,i}}^2\|u\|^2_{\tilde H^3(\Omega)}+h_{K_{B,i}}\|u\|^2_{\tilde
H^2(K_{B,i})}\right).\label{eq: bound on interface elems}
\end{eqnarray}
Due to the shape regularity of mesh $\mathcal{T}_h$, we have the following bound on the union of interface elements
\begin{equation}\label{eq: bound on interface elems sum}
  \sum_{K\in \mathcal{T}_h^i} h_K^2\|u\|_{\tilde H^3(\Omega)}^2 \leq h \|u\|_{\tilde H^3(\Omega)}^2\sum_{K\in \mathcal{T}_h^i} h_K \leq C h \|u\|_{\tilde H^3(\Omega)}^2.
\end{equation}
Summing up the estimates \eqref{eq: bound on noninterface elems} and \eqref{eq: bound on interface elems} over all interior edges, and using the bound in \eqref{eq: bound on interface elems sum}, we obtain
\begin{equation}\label{eq: DG error temp11}
  \sum\limits_{B\in\mathcal{\mathring E}_h}\sum_{i=1,2}\|(\beta\nabla
(u-I_hu)\cdot\mathbf{n}_B)|_{K_{B,i}}\|^2_{L^2(B)} \leq Ch\|u\|^2_{\tilde H^3(\Omega)}+Ch\|u\|^2_{\tilde H^2(\Omega)} \leq Ch\|u\|^2_{\tilde H^3(\Omega)}.
\end{equation}
Then, substituting \eqref{eq: DG error temp11} to \eqref{eq: DG error temp10} we obtain
\begin{equation}\label{eq: DG error temp12}
\|u_h-I_hu\|_h\leq Ch\|u\|_{\tilde H^3(\Omega)}.
\end{equation}
Finally, the error estimate \eqref{eq: DG error estimates} follows from triangle inequality, \eqref{eq: DG error temp12} and \eqref{eq: interpolation h norm}.
\end{proof}

\begin{remark}
 The DG-IFE methods proposed in this article and their related error estimation can be extended to arbitrary shape-regular unstructured interface independent meshes.
\end{remark}
\begin{remark}
 The proof of Theorem 3.1 requires that the solution is piecewise $H^3$, which is higher than the usual piecewise $H^2$ assumption imposed on methods using a finite element space
 based on linear polynomials. Consequently, our error estimate here is optimal according to the rate of convergence expected from linear polynomials but not with respect to the regularity of solution space.
\end{remark}

\section{Numerical Examples}

In this section, we present numerical examples to demonstrate features of interior penalty DG-IFE methods for elliptic interface problems. Let the solution domain $\Omega$ be the open rectangle $(-1,1)\times(-1,1)$ and let the interface $\Gamma$ be the ellipse centered at $(x_0,y_0) = (-0.2,0.1)$ with semi-axes $a = \frac{\pi}{6.28}$,
$b = \frac{3}{2}a$. The interface separates $\Omega$ into two sub-domains, denoted by $\Omega^-$ and $\Omega^+$, \emph{i.e.},
\begin{equation*}
    \Omega^- = \{(x,y): r(x,y) < 1\}, ~~~\text{and}~~~
    \Omega^+ = \{(x,y): r(x,y) > 1\},
\end{equation*}
where
\begin{equation*}
  r(x,y) = \sqrt{\frac{(x-x_0)^2}{a^2} + \frac{(y-y_0)^2}{b^2}}.
\end{equation*}
The exact solution $u$ to the interface problem is chosen as follows
\begin{equation}\label{eq: true solution ellipse}
    u(x,y) =
    \left\{
      \begin{array}{ll}
        a^2b^2\frac{r^p}{\beta^-}, & \text{if~} (x,y) \in\Omega^-, \\
        a^2b^2\left(\frac{r^p}{\beta^+} + \frac{1}{\beta^-} - \frac{1}{\beta^+}\right), &\text{if~} (x,y)\in\Omega^+.
      \end{array}
    \right.
\end{equation}
Here $p$ is a parameter and we choose $p = 5$ in Examples 1 - 3 representing a solution with enough regularity, and $p=0.5$ in Example 4 representing a solution with low regularity. Note that this solution does not satisfy the homogeneous boundary condition \eqref{eq: bc}. We use this function for numerical verification because both the algorithm and the analysis in Section 2 and Section 3 can be extended to the nonhomogeneous boundary condition case via a standard treatment. \\

{\bf Example 1}: In this example, we present a group of numerical results for demonstrating the convergence of the DG-IFE methods on Cartesian triangular meshes. Additional numerical results on rectangular meshes are provided in \cite{XHe_Thesis_Bilinear_IFE,XHe_TLin_YLin_DG,He_Lin_Lin_Selective_DG}. Specifically,
the Cartesian triangular meshes $\{\mathcal{T}_h, h>0\}$ are formed by first partitioning $\Omega$ into $N\times N$ congruent squares of size $h = 2/N$ and then dividing each rectangle into two congruent triangles with one of its diagonal lines.

First, we consider the case in which $(\beta^-,\beta^+) = (1,10)$ representing a moderate discontinuity in the diffusion coefficient across the interface. The symmetric DG-IFE scheme is employed to solve the interface problem with parameters $\alpha = 1$ and $\sigma_B^0 =1000$ for all interior edges.  Errors of numerical solutions in the $L^\infty$, $L^2$, and semi-$H^1$ norms are reported in Table \ref{table: SIPG IFE error 1 10 penalty 1000 Ellipse Linear}. For comparison, we also solve the same interface problem using the continuous Galerkin linear IFE method \cite{ZLi_TLin_YLin_RRogers_linear_IFE,ZLi_TLin_XWu_Linear_IFE} on the same meshes, and the corresponding numerical results are listed in Table \ref{table: Galerkin IFE error 1 10 Ellipse Linear}.

\begin{table}[h]
\begin{center}
\begin{tabular}{|c|cc||cc||cc|}
\hline
$N$
& $\|\cdot\|_{L^\infty}$ & rate
& $\|\cdot\|_{L^2}$ & rate
& $|\cdot|_{H^1}$ & rate\\
 \hline
$10$      &$2.8333E{-2}$ &         &$3.7991E{-2}$ &         &$6.7917E{-1}$ &         \\
$20$      &$8.4503E{-3}$ & 1.7454  &$9.3605E{-3}$ & 2.0210  &$3.4653E{-1}$ & 0.9708  \\
$40$      &$2.5075E{-3}$ & 1.7527  &$2.3062E{-3}$ & 2.0210  &$1.7456E{-1}$ & 0.9893  \\
$80$      &$7.2318E{-4}$ & 1.7938  &$5.6970E{-4}$ & 2.0173  &$8.7630E{-2}$ & 0.9942  \\
$160$     &$2.0134E{-4}$ & 1.8447  &$1.4140E{-4}$ & 2.0105  &$4.3903E{-2}$ & 0.9971  \\
$320$     &$5.4720E{-5}$ & 1.8795  &$3.5178E{-5}$ & 2.0070  &$2.1972E{-2}$ & 0.9986  \\
$640$     &$1.4450E{-5}$ & 1.9210  &$8.7729E{-6}$ & 2.0035  &$1.0991E{-2}$ & 0.9994  \\
$1280$    &$3.7496E{-6}$ & 1.9463  &$2.1903E{-6}$ & 2.0019  &$5.4965E{-3}$ & 0.9997  \\
\hline
\end{tabular}
\end{center}
\caption{Errors of linear DG-IFE solutions with $\beta^- = 1$, $\beta^+ = 10$, $\sigma_B^0 = 1000$.}
\label{table: SIPG IFE error 1 10 penalty 1000 Ellipse Linear}
\end{table}

\begin{table}[h]
\begin{center}
\begin{tabular}{|c|cc||cc||cc|}
\hline
$N$
& $\|\cdot\|_{L^\infty}$ & rate
& $\|\cdot\|_{L^2}$ & rate
& $|\cdot|_{H^1}$ & rate\\
 \hline
$10$      &$2.8619E{-2}$ &         &$4.4051E{-2}$ &         &$6.7876E{-1}$ &         \\
$20$      &$1.1416E{-2}$ & 1.3259  &$1.1333E{-2}$ & 1.9587  &$3.4808E{-1}$ & 0.9635  \\
$40$      &$5.3027E{-3}$ & 1.1062  &$2.8882E{-3}$ & 1.9722  &$1.7641E{-1}$ & 0.9805  \\
$80$      &$1.9396E{-3}$ & 1.4510  &$7.3078E{-4}$ & 1.9827  &$8.9155E{-2}$ & 0.9845  \\
$160$     &$1.0689E{-3}$ & 0.8596  &$1.8726E{-4}$ & 1.9644  &$4.5387E{-2}$ & 0.9740  \\
$320$     &$5.4774E{-4}$ & 0.9646  &$5.1220E{-5}$ & 1.8702  &$2.3305E{-2}$ & 0.9617  \\
$640$     &$2.7498E{-4}$ & 0.9942  &$1.6771E{-5}$ & 1.6107  &$1.2277E{-2}$ & 0.9247  \\
$1280$    &$1.4113E{-4}$ & 0.9623  &$7.1759E{-6}$ & 1.2248  &$6.7321E{-3}$ & 0.8668  \\
\hline
\end{tabular}
\end{center}
\caption{Errors of the classic Galerkin IFE solutions with $\beta^- = 1$, $\beta^+ = 10$.}
\label{table: Galerkin IFE error 1 10 Ellipse Linear}
\end{table}

The data in Table \ref{table: SIPG IFE error 1 10 penalty 1000 Ellipse Linear} clearly demonstrate that the convergence rate of the DG-IFE method in the semi-$H^1$ norm is optimal which corroborates the {\it a priori} error estimates \eqref{eq: DG error estimates} for DG-IFE methods since the semi-$H^1$ norm is part of the energy norm. In addition, the data in this table indicate that the convergence rate of the DG-IFE method in the $L^2$ norm is also optimal. However, from Table
\ref{table: Galerkin IFE error 1 10 Ellipse Linear}, we can see that the convergence rates of the Galerkin IFE solution in the $L^2$ and $H^1$ norms start to deteriorate when the mesh size becomes smaller than $h = 2/320$. This comparison indicates that the DG-IFE methods are more stable than the continuous Galerkin IFE method.

Next we consider the case involving a larger discontinuity in the diffusion coefficient, \emph{i.e.}, $\beta^- =1$, and $\beta^+ = 1000$. We use nonsymmetric DG-IFE scheme for this experiment and choose $\sigma_B^0 = 1000$ for all interior edges. As demonstrated by the data in Table \ref{table: NIPG IFE error 1 1000 penalty 1000 Ellipse Linear}, the DG-IFE solutions converge optimally in the $L^2$ and semi-$H^1$ norms.
\\

\begin{table}[h]
\begin{center}
\begin{tabular}{|c|cc||cc||cc|}
\hline
$N$
& $\|\cdot\|_{L^\infty}$ & rate
& $\|\cdot\|_{L^2}$ & rate
& $|\cdot|_{H^1}$ & rate\\
 \hline
$10$      &$1.9381E{-2}$ &         &$1.5338E{-2}$ &         &$2.1012E{-1}$ &         \\
$20$      &$1.2420E{-2}$ & 0.6420  &$6.0704E{-3}$ & 1.3373  &$1.3141E{-1}$ & 0.6772  \\
$40$      &$4.0332E{-3}$ & 1.6227  &$1.4957E{-3}$ & 2.0210  &$6.9522E{-2}$ & 0.9185  \\
$80$      &$9.9934E{-4}$ & 2.0129  &$3.6124E{-4}$ & 2.0498  &$3.5490E{-2}$ & 0.9701  \\
$160$     &$2.7965E{-4}$ & 1.8374  &$8.9863E{-5}$ & 2.0072  &$1.7949E{-2}$ & 0.9835  \\
$320$     &$8.0700E{-5}$ & 1.7930  &$2.1864E{-5}$ & 2.0392  &$9.0223E{-3}$ & 0.9924  \\
$640$     &$2.2017E{-5}$ & 1.8740  &$5.3914E{-6}$ & 2.0198  &$4.5229E{-3}$ & 0.9963  \\
$1280$    &$5.9615E{-6}$ & 1.8848  &$1.3343E{-6}$ & 2.0146  &$2.2641E{-3}$ & 0.9983  \\
\hline
\end{tabular}
\end{center}
\caption{Errors of linear DG-IFE solutions with $\beta^- = 1$, $\beta^+ = 1000$, $\sigma_B^0 = 1000$.}
\label{table: NIPG IFE error 1 1000 penalty 1000 Ellipse Linear}
\end{table}

{\bf Example 2}: From Tables \ref{table: SIPG IFE error 1 10 penalty 1000 Ellipse Linear} and \ref{table: Galerkin IFE error 1 10 Ellipse Linear}, it is interesting to note that errors in the DG-IFE solutions gauged in $L^\infty$ norm are much smaller than those in the classic Galerkin IFE solutions when the mesh size is small enough.
It has been observed that the classic IFE solution has a so called ``crown" shortcoming as demonstrated by the plot on the left side of Figure \ref{fig: SIPDG IFE error 1 10 ellipse}, meaning its point-wise accuracy is much worse in the vicinity of the interface than the rest of the solution domain. We think this severity of inaccuracy is caused by the discontinuity in the IFE functions across the interface edges. Nevertheless, the DG-IFE methods contain penalty terms that can alleviate the adverse impacts from the
discontinuity across element edges, especially those from interface edges. Therefore, DG-IFE methods can usually outperform the classic Galerkin IFE method around the interface as demonstrated by the plot on the right side of Figure \ref{fig: SIPDG IFE error 1 10 ellipse}. IFE solutions in these plots are generated on a mesh formed by partitioning $\Omega$ into $160 \times 160$ congruent rectangles first, then generating triangular elements by the diagonal line of these rectangles.\\

\begin{figure}[htb]
  \centering
  \includegraphics[width=.49\textwidth]{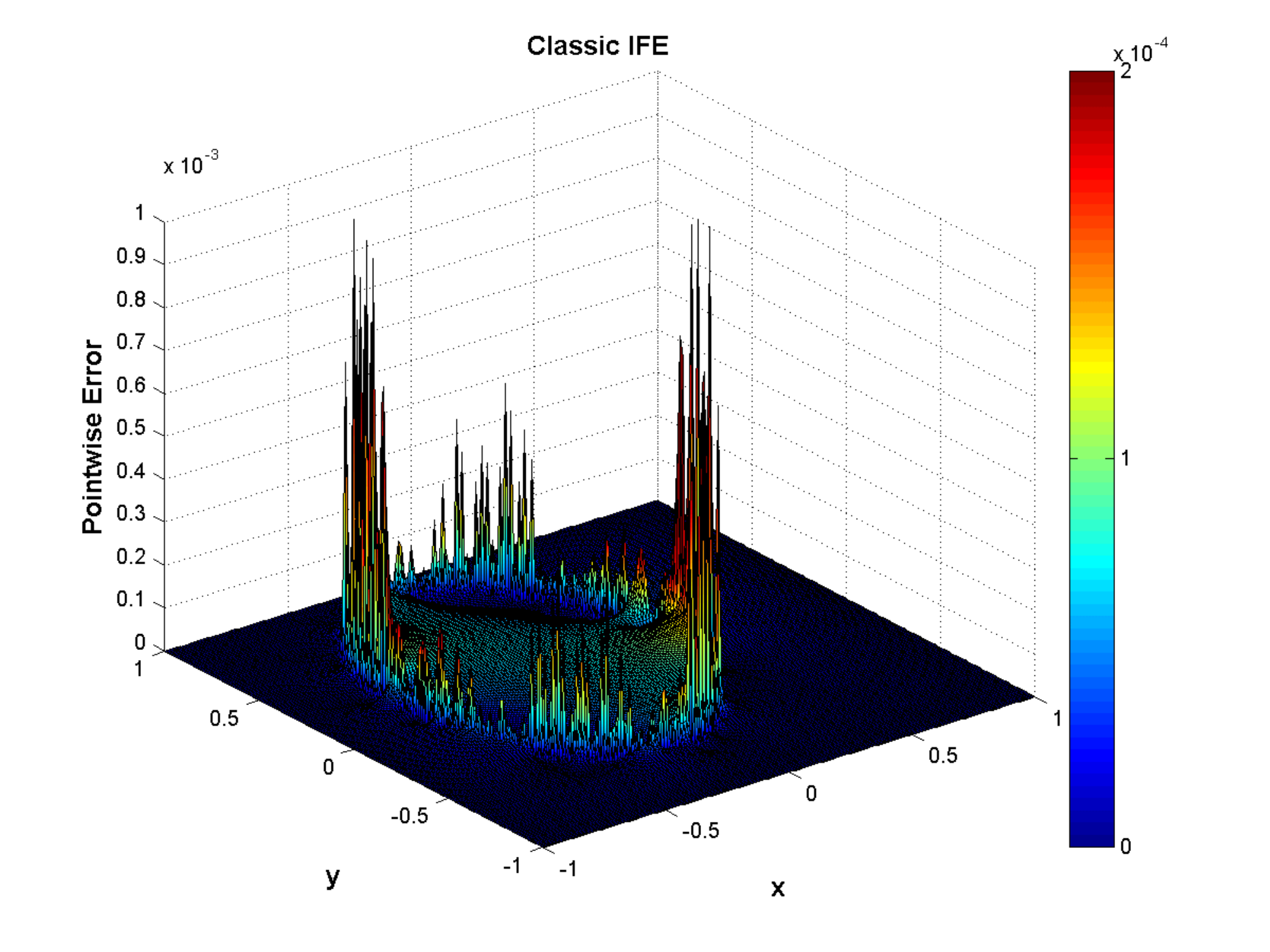}
  \includegraphics[width=.49\textwidth]{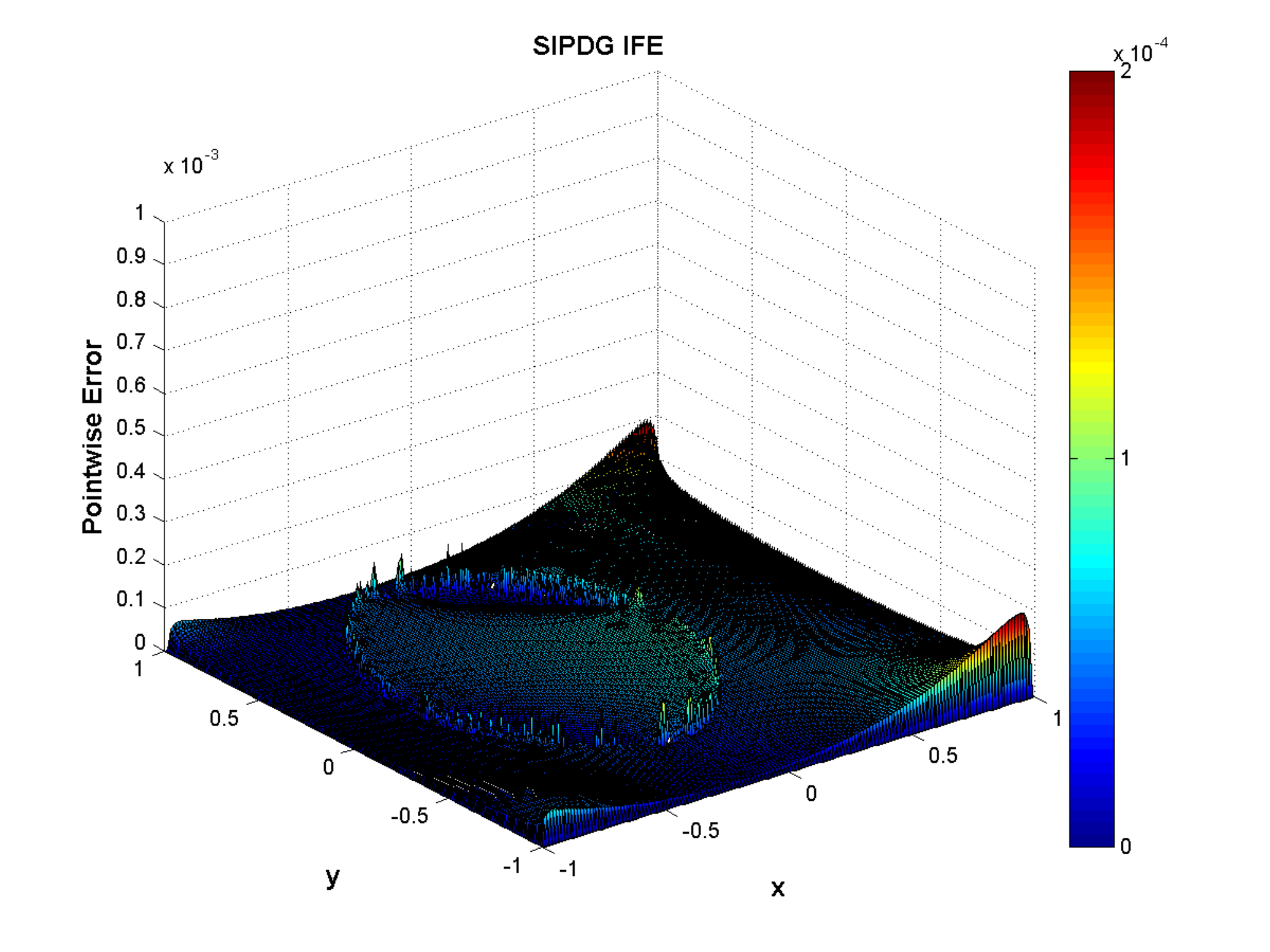}\\
  \caption{Point-wise error of Galerkin IFE solution and DG-IFE solution.}
  \label{fig: SIPDG IFE error 1 10 ellipse}
\end{figure}

{\bf Example 3}: One desirable feature of the DG formulation is the local adaptivity in mesh or polynomials because this formulation does not require the inter-element continuity of finite element functions. The combination of IFE spaces and the DG formulation leads to a new class of finite element methods that allow local mesh refinement while maintaining the possibility of using the desirable structured Cartesian meshes for solving problems with nontrivial interface geometry. This example is for demonstrating this feature of the DG-IFE methods.

The discontinuity in the coefficient of the interface problem limits the smoothness of the exact solution around interface. The low regularity of the solution around interface usually has an adverse impact on the accuracy of the numerical solution around the interface. To overcome this challenge, one can employ more finite element functions around the interface and a way to achieve this is to refine interface elements. To demonstrate this idea, we consider the same example described above for a larger coefficient jump ($\beta^- =1$, $\beta^+ = 1000$). We start with a uniform Cartesian mesh $\mathcal{T}_h^{(0)}$ consisting of $10\times 10$ rectangles. For $k \geq 1$, the mesh $\mathcal{T}_h^{(k)}$ is generated by refining the previous mesh $\mathcal{T}_h^{(k-1)}$ via cutting each of its interface elements into four congruent rectangles by connecting midpoints on opposite edges. We then solve the
interface problem \eqref{eq: pde} - \eqref{eq: jump2} on a sequence of 6 such meshes generated by the refinement procedure described above using the nonsymmetric DG-IFE method with $\sigma_B^0 = 1000$ on all internal edges. Errors in the $L^\infty$, $L^2$, and semi-$H^1$ norms generated on each mesh are presented in Table \ref{table: local refinement error}. In the second column, $|\mathcal{T}_h^{(k)}|$ denotes the number of elements in the mesh $\mathcal{T}_h^{(k)}$. The numbers of degrees of freedom are listed in the third column. The initial mesh $\mathcal{T}_h^{(0)}$ and the refined meshes $\mathcal{T}_h^{(2)}$ and $\mathcal{T}_h^{(4)}$ are illustrated in Figure \ref{fig: mesh refinement}.

\begin{figure}[ht]
  \centering
  \includegraphics[width=.25\textwidth]{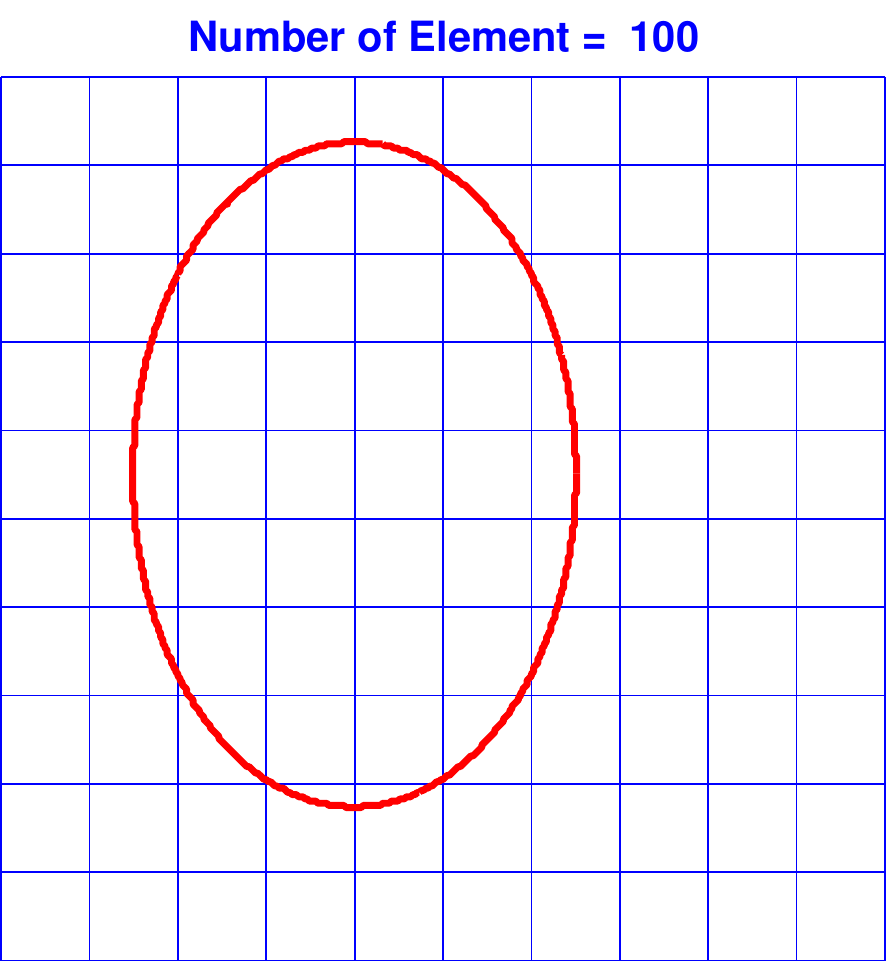}~~~~
  \includegraphics[width=.25\textwidth]{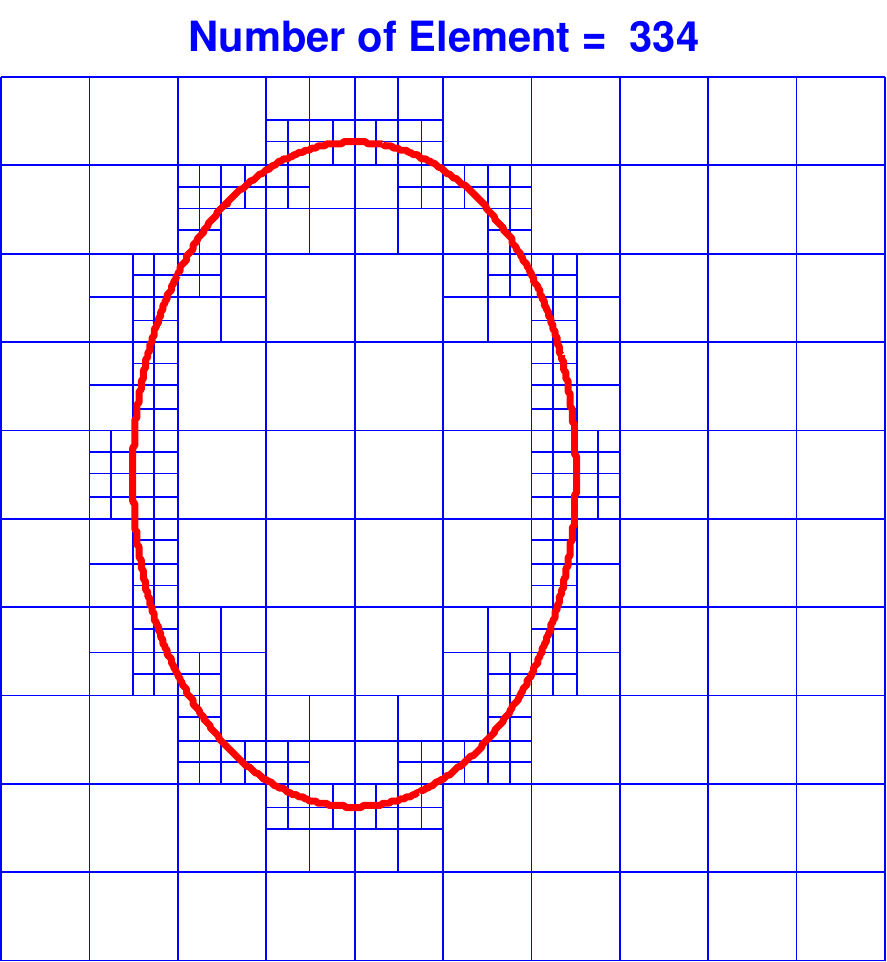}~~~~
  \includegraphics[width=.25\textwidth]{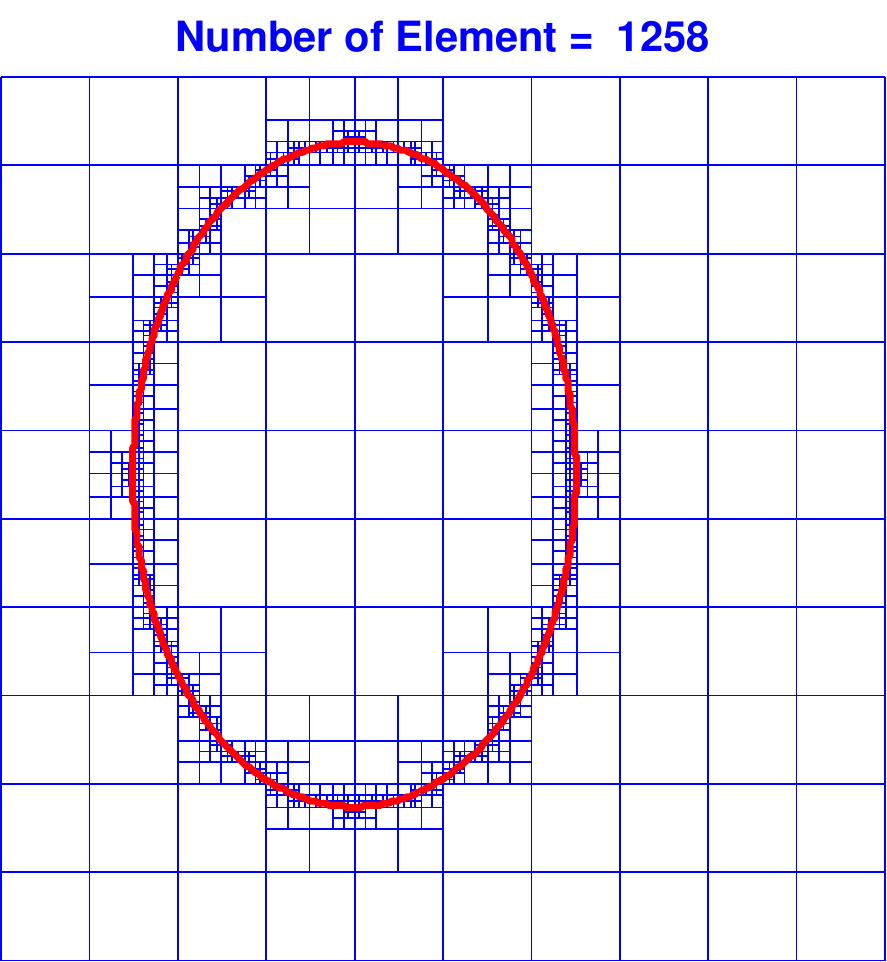}\\
  \caption{Solution meshes $\mathcal{T}_h^{(0)}$, $\mathcal{T}_h^{(2)}$, and $\mathcal{T}_h^{(4)}$.}
  \label{fig: mesh refinement}
\end{figure}

\begin{table}[ht]
\begin{center}
\begin{tabular}{|c|c|c|c|c|c|}
\hline
Mesh & $|\mathcal{T}_h^{(i)}|$ & $DoF$
& $\|\cdot\|_{L^\infty}$
& $\|\cdot\|_{L^2}$
& $|\cdot|_{H^1}$ \\
 \hline
$\mathcal{T}_h^{(0)}$ & 100 & 400  & 1.5064E{-2} & 1.5667E{-2} & 1.9393E{-1} \\
$\mathcal{T}_h^{(1)}$ & 178 & 712  & 1.9242E{-2} & 8.7358E{-3} & 1.5041E{-1} \\
$\mathcal{T}_h^{(2)}$ & 334 &1336  & 1.2799E{-2} & 5.9608E{-3} & 1.2761E{-1} \\
$\mathcal{T}_h^{(3)}$ & 646 &2584  & 1.2789E{-2} & 5.7452E{-3} & 1.2321E{-1} \\
$\mathcal{T}_h^{(4)}$ &1258 &5032  & 1.2484E{-2} & 5.6245E{-3} & 1.2233E{-1} \\
$\mathcal{T}_h^{(5)}$ &2470 &9880  & 1.2431E{-2} & 5.5988E{-3} & 1.2213E{-1} \\
$\mathcal{T}_h^{(6)}$ &4882 &19528 & 1.2424E{-2} & 5.5940E{-3} & 1.2209E{-1} \\
\hline
\end{tabular}
\end{center}
\caption{Errors of NIPDG-IFE solutions on meshes with local refinement}
\label{table: local refinement error}
\end{table}

The data in Table \ref{table: local refinement error} show that the global error in the $L^2$ and semi-$H^1$ norms are significantly reduced in the first two steps of local mesh refinements. But further refinements performed on interface elements do not reduce the global error as much as the first two steps. We believe this is because, after the first two refinements, errors of the DG-IFE solutions over non-interface elements become more significant. To further increase the accuracy of DG-IFE solutions, refinement on non-interface elements is also necessary. To demonstrate this idea, we simulate an adaptive local mesh refinement over the whole solution domain. Since {\it a posteriori } error estimators for IFE methods are not available yet, we use the actual error as the ``ideal" error indicator to guide the local refinement just for a proof of concept.

We start from a uniform Cartesian mesh $\mathcal{T}_h^{(0)}$ consisting of $10\times 10$ rectangles. For $k \geq 0$, we produce a DG-IFE solution $u_h$ for the interface problem on the mesh $\mathcal{T}_h^{(k)}$ and compute the local semi-$H^1$ norm error $|u-u_h|_{H^1(T)}$ on each element of $\mathcal{T}_h^{(k)}$. We sort these local errors from the largest to the smallest and use this order to form the smallest collection $\tilde{\mathcal{T}}_h^{(k)}$ of the first few elements such that
\begin{equation}\label{eq: refine rule}
  \sum_{T\in \tilde{\mathcal{T}}_h^{(k)}}|u-u_h|^2_{H^1(T)} \geq \theta |u-u_h|_{H^1(\Omega)}^2.
\end{equation}
Then, we generate a new mesh $\mathcal{T}_h^{(k+1)}$ by refining $\mathcal{T}_h^{(k)}$ via cutting each of the elements in $\tilde{\mathcal{T}}_h^{(k)}$ into four congruent rectangles by connecting midpoints on its opposite edges. The computation repeats over the new mesh $\mathcal{T}_h^{(k+1)}$.

We choose $\theta = 0.2$ and conduct adaptive DG-IFE scheme on each locally refined mesh. In Figure \ref{fig: NIPDG IFE error 1 1000}, we show the initial mesh $\mathcal{T}_h^{(0)}$, the refined meshes $\mathcal{T}_h^{(7)}$,$\mathcal{T}_h^{(12)}$, and $\mathcal{T}_h^{(17)}$, and errors of the DG-IFE solutions generated on those meshes. From these plots, we can easily see that the global error in the DG-IFE solution is reduced as the local mesh refinement automatically deploys smaller elements at locations needed according to the ``ideal" error indicator while maintaining the Cartesian structure of the meshes.

To see the effectiveness of the adaptive DG-IFE methods, we compare the errors in DG-IFE solutions generated via adaptive mesh refinement to errors of DG-IFE solutions generated on uniform meshes with comparable degrees of freedom in Figure \ref{fig: convergence of mesh refinement}. The semi-$H^1$ norm errors presented in the right plot in Figure \ref{fig: convergence of mesh refinement} shows that the magnitude of errors in adaptive DG-IFE method are smaller than in the method with a uniform mesh when their degrees of freedom are comparable. However, the order of convergence for both schemes are optimal by comparing their errors with the reference line of slope $-1/2$ (the same criteria is used in \cite{Cai_Ye_Zhang_DG_Interface}).
We note that the exact solution $u$ defined in \eqref{eq: true solution ellipse} with $p=5$ is piecewisely smooth, \emph{i.e.} $u\in \tilde H^3(\Omega)$ (see the left plot in Figure \ref{fig: convergence of mesh refinement}), although the global regularity is impacted by the discontinuity of the coefficients. The optimal convergence of the numerical solution in uniform mesh refinement confirms our theoretical error estimate \eqref{eq: DG error estimates} in Section 3.

\begin{figure}[ht]
  \centering
  \includegraphics[width=.25\textwidth]{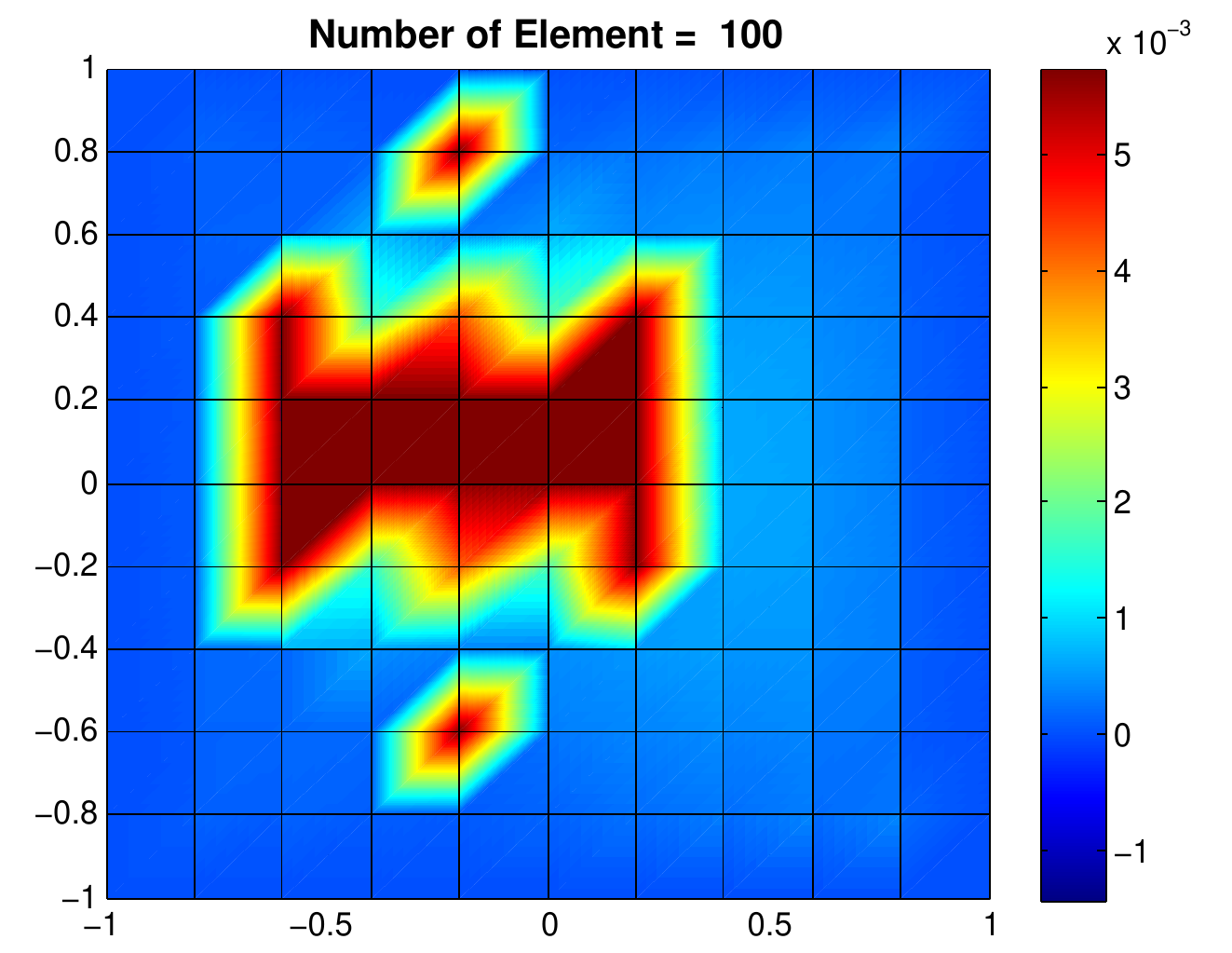}\!\!\!\!\!\!
  \includegraphics[width=.25\textwidth]{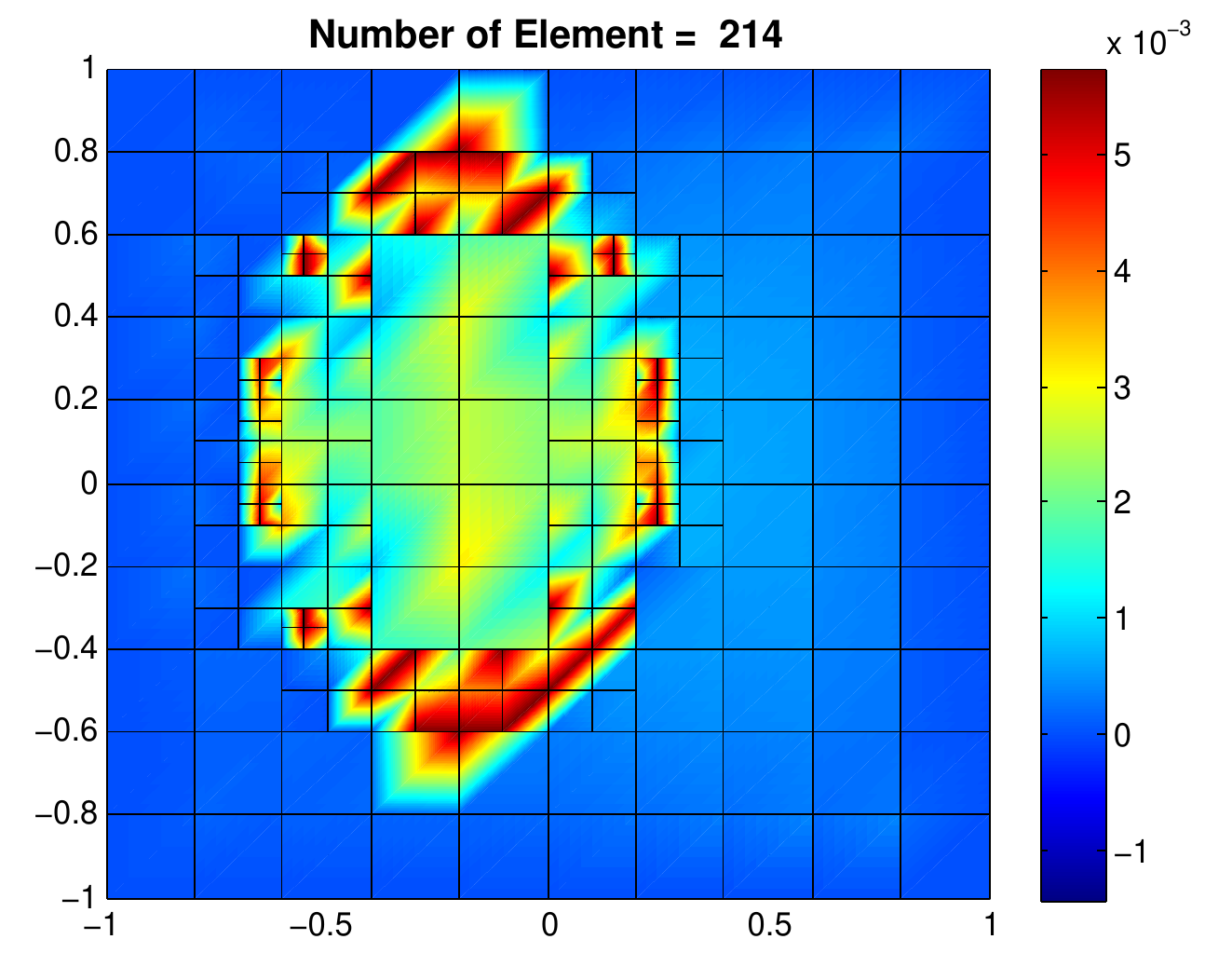}\!\!\!\!\!\!
  \includegraphics[width=.25\textwidth]{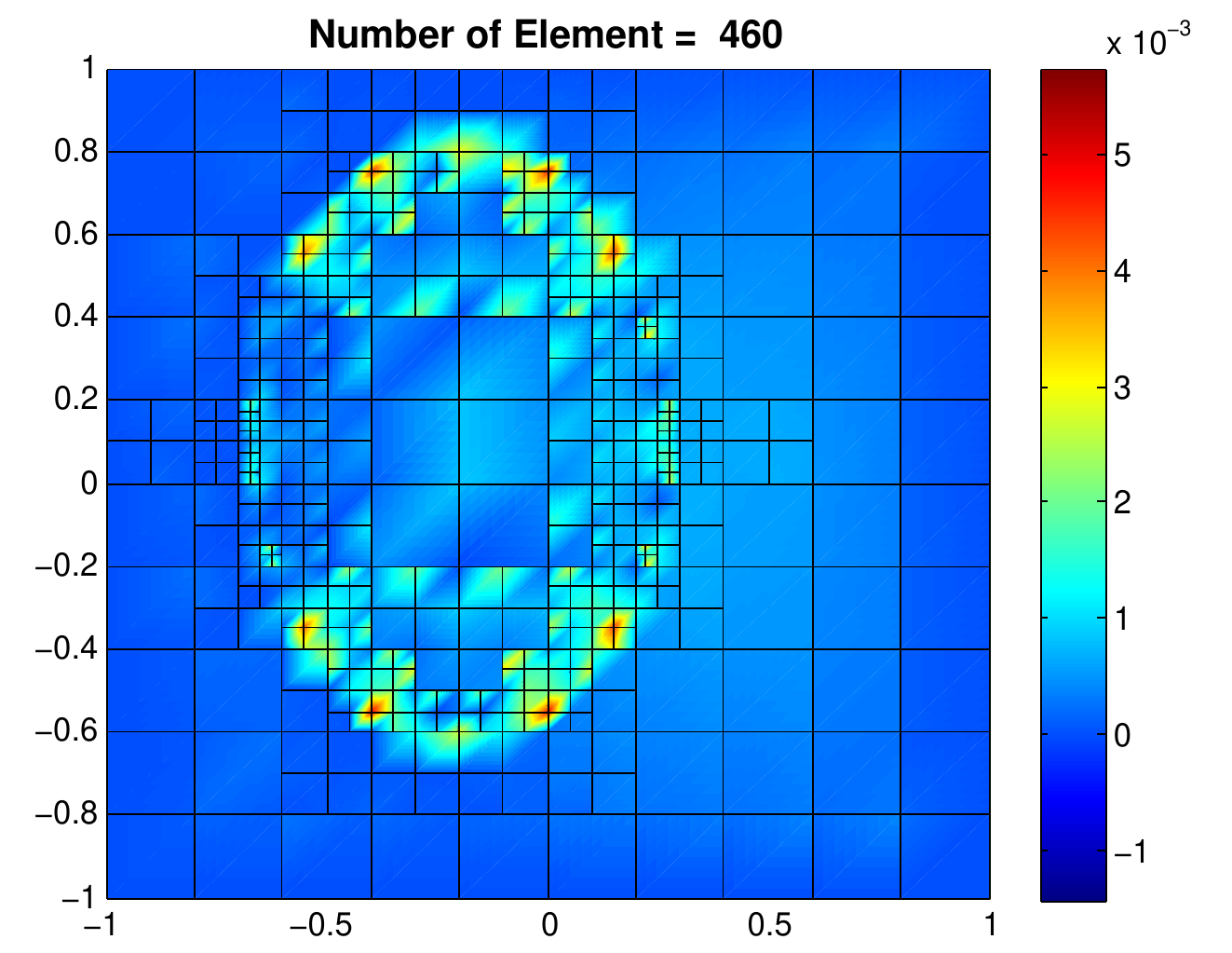}\!\!\!\!\!\!
  \includegraphics[width=.25\textwidth]{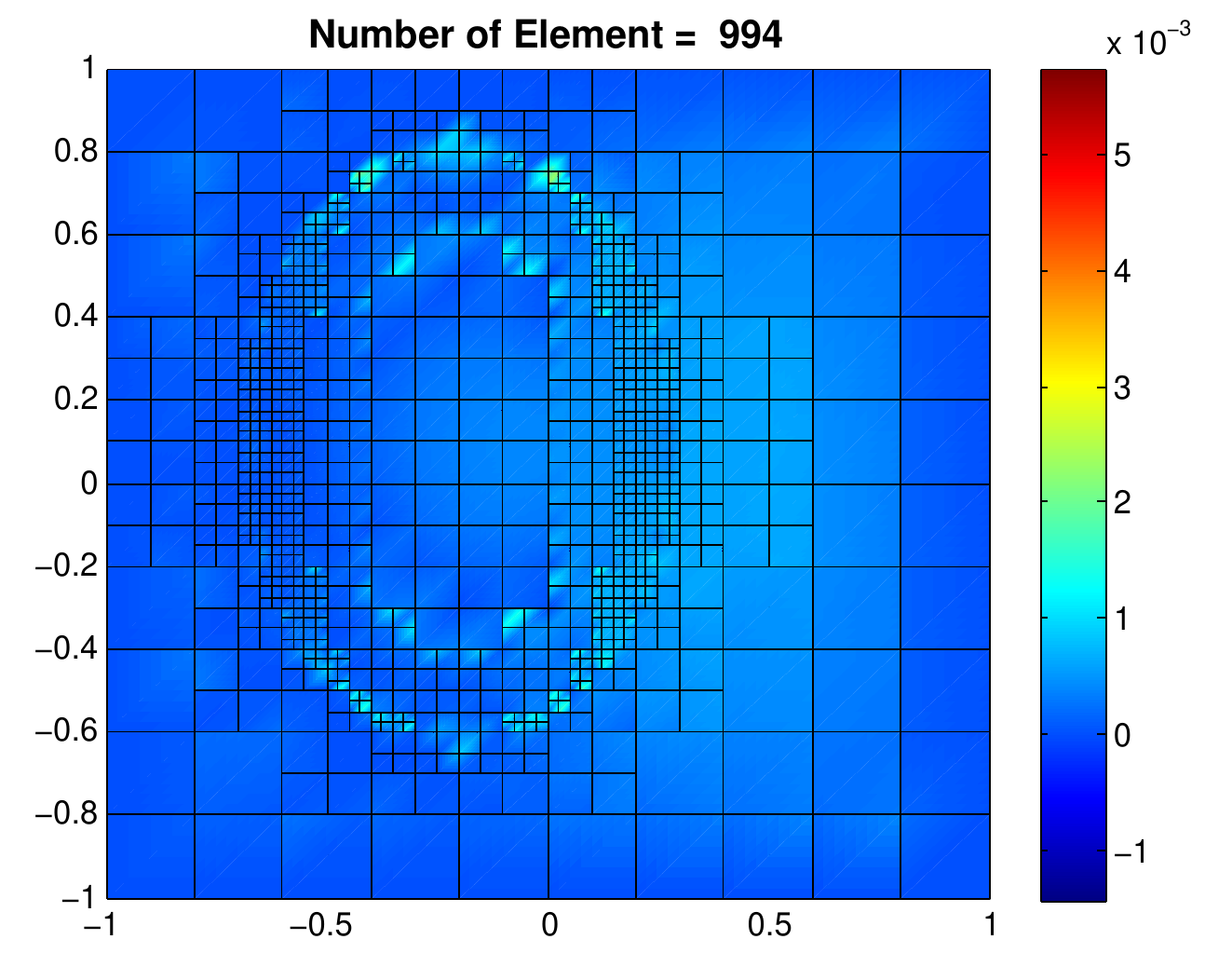}\\
  \caption{Point-wise error of DG-IFE solutions on meshes: $\mathcal{T}_h^{(0)}$, $\mathcal{T}_h^{(7)}$,$\mathcal{T}_h^{(12)}$, and $\mathcal{T}_h^{(17)}$}
  \label{fig: NIPDG IFE error 1 1000}
\end{figure}

\begin{figure}[htb]
  \centering
  \includegraphics[width=.4\textwidth]{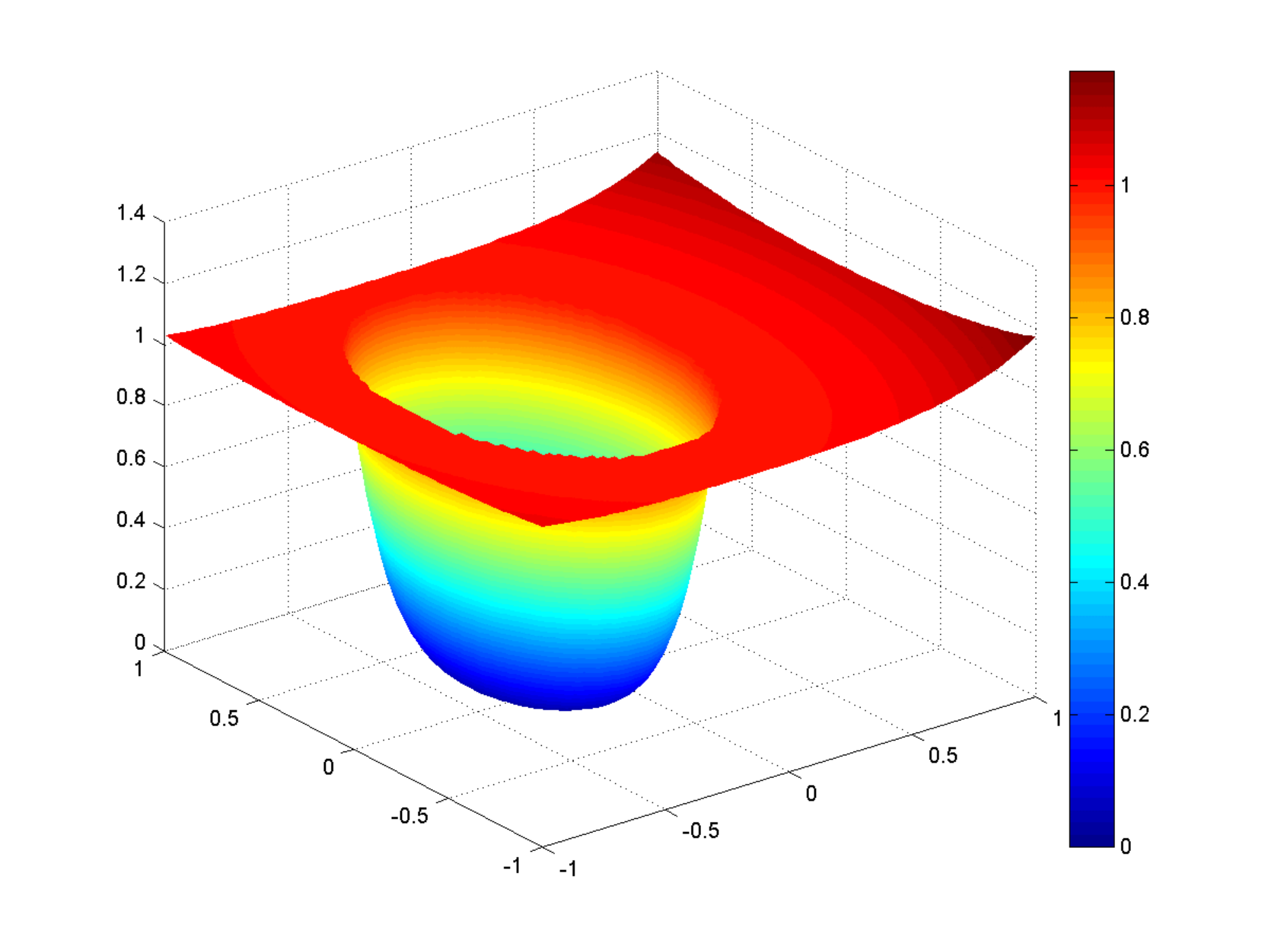}~~~
  \includegraphics[width=.4\textwidth]{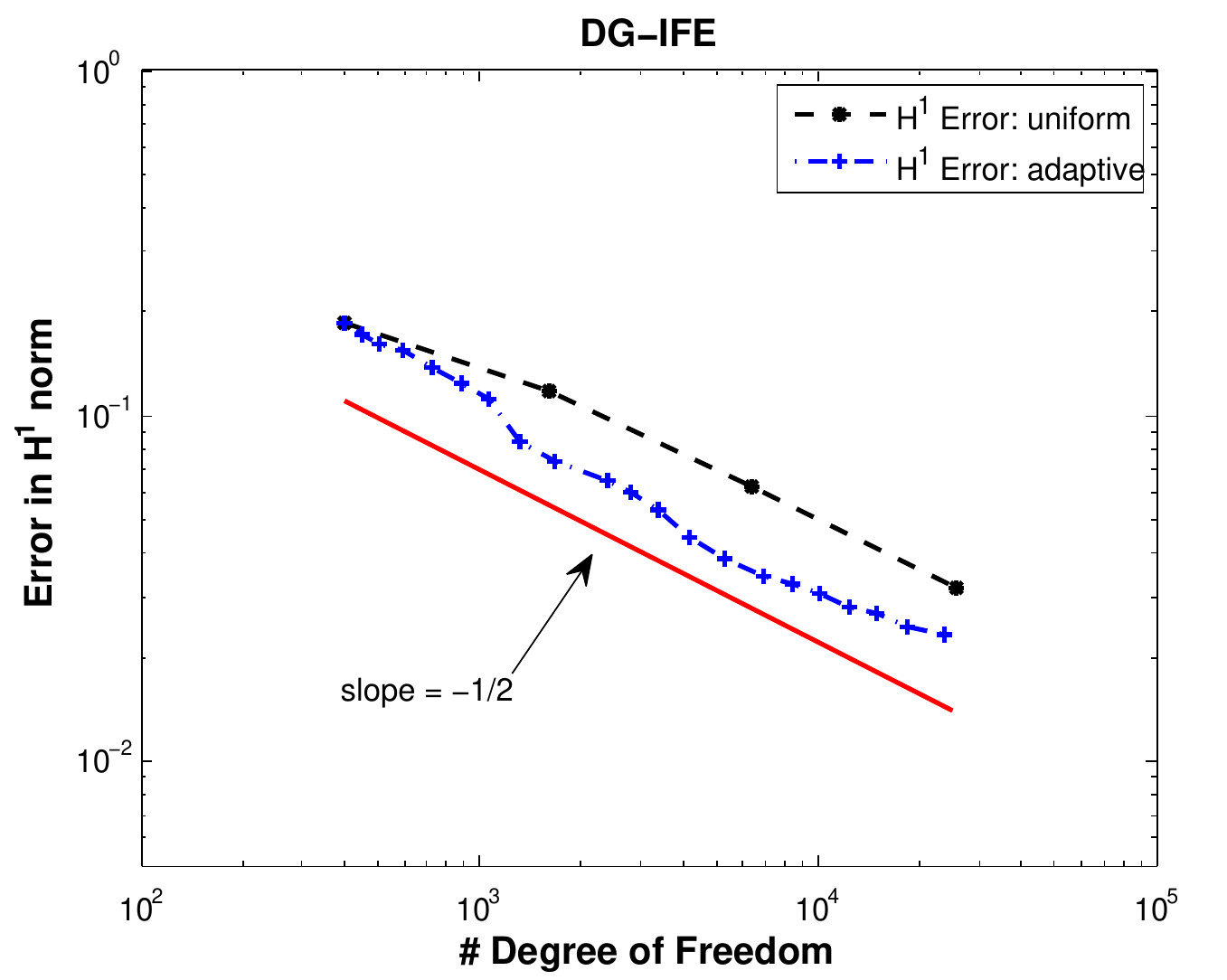}
  \\
  \caption{Left: the exact solution $u$ in Example 3. Right: a log-log scale plot of errors of in numerical solutions generated by the adaptive and uniform mesh refinement.
  The red line with a $-0.5$ slope is for the reference of optimal convergence.}
  \label{fig: convergence of mesh refinement}
\end{figure}

{\bf Example 4}: In many applications, because of the insufficient regularity in the data, solutions to the involved boundary value problems may not be piecewisely smooth enough for a certain convergence theorem to hold. In such cases, finite element or DG methods based on uniform mesh refinement usually fail to converge optimally, but adaptive FE or DG methods with suitably designed mesh refinement strategies can still generate optimally convergent numerical solutions \cite{Cai_Ye_Zhang_DG_Interface, Cai_Zhang_Flux_Recovery}. For interface problems with discontinuous coefficients, the phenomenon is similar. In this example, we demonstrate behaviors of DG-IFE method in adaptive and uniform mesh refinements for solving interface problems whose exact solution has singularity. In particular, this example indicates that the DG-IFE method with adaptive refinement on interface independent meshes can satisfactorily handle interface problems whose exact solutions are less smooth.

We choose the exact solution $u$ in the form of \eqref{eq: true solution ellipse} with $p=0.5$ such that it does not satisfy the regularity condition required by Theorem \ref{th:DG_convergence} on the convergence of the DG-IFE method. This lack of regularity is caused by the singularity of the exact solution at the center of the ellipse as depicted in the left plot of Figure \ref{fig: exactu p=0.5}. The coefficients are set to be $\beta^-=1$, $\beta^+ = 10$. First, we present the data generated by the DG-IFE method with uniform mesh refinement in Table \ref{table: NIPG IFE error 1 10 penalty 100 Ellipse BiLinear p0.5}. Specifically, these data are produced by the nonsymmetric DG-IFE method with the penalty $\sigma_B^0 = 100$ on every edge. Errors in $L^\infty$, $L^2$, and semi-$H^1$ norms are reported, and the convergence rates
of the DG-IFE method are obviously not optimal in all three corresponding norms.

\begin{table}[h]
\begin{center}
\begin{tabular}{|c|cc||cc||cc|}
\hline
$N$
& $\|\cdot\|_{L^\infty}$ & rate
& $\|\cdot\|_{L^2}$ & rate
& $|\cdot|_{H^1}$ & rate\\
 \hline
$10$      &$4.2599E{-2}$ &         &$1.8474E{-2}$ &         &$1.0415E{-1}$ &         \\
$20$      &$5.7301E{-2}$ & -.4278  &$8.0329E{-3}$ & 1.2015  &$5.6787E{-2}$ & 0.8751  \\
$40$      &$4.5270E{-2}$ & 0.3400  &$5.5008E{-3}$ & 0.5463  &$4.2830E{-2}$ & 0.4069  \\
$80$      &$3.5386E{-2}$ & 0.3554  &$3.8353E{-3}$ & 0.5203  &$3.2011E{-2}$ & 0.4201  \\
$160$     &$2.7412E{-2}$ & 0.3684  &$2.6964E{-3}$ & 0.5083  &$2.3774E{-2}$ & 0.4292  \\
$320$     &$2.1075E{-2}$ & 0.3793  &$1.9024E{-3}$ & 0.5032  &$1.7573E{-2}$ & 0.4361  \\
$640$     &$1.6098E{-2}$ & 0.3886  &$1.3441E{-3}$ & 0.5012  &$1.2940E{-2}$ & 0.4415  \\
$1280$    &$1.2237E{-2}$ & 0.3967  &$9.5014E{-4}$ & 0.5005  &$9.5021E{-3}$ & 0.4461  \\
\hline
\end{tabular}
\end{center}
\caption{Errors of bilinear nonsymmetric DG-IFE solutions with $p=0.5$}
\label{table: NIPG IFE error 1 10 penalty 100 Ellipse BiLinear p0.5}
\end{table}

Next, we report the performance of the adaptive DG-IFE method for solving the same interface problem with this less smooth exact solution. As in Example 3, we use the exact error as an ``ideal" error indicator for mesh refinement just for a proof of concept. Starting with a uniform mesh $\mathcal{T}_h^{(0)}$ consisting of $10\times 10$ rectangles, we perform the local mesh refinement based on the same rule as the one in \eqref{eq: refine rule} with the threshold $\theta = 0.2$.
Errors in semi-$H^1$ norm are depicted in Figure \ref{fig: exactu p=0.5}, in which, as a comparison, errors from uniform mesh refinement are also plotted. It is obvious that the adaptive DG-IFE method is far more accurate than
the DG-IFE method based uniform meshes when their numbers of degrees of freedom are comparable. Moreover, comparing the errors with the reference line of slope $-1/2$, it is obvious that the rate of convergence of the adaptive DG-IFE method is close to optimal from the point view of the degrees of freedom while the rate of convergence of
DG-IFE method based on uniform mesh is not optimal. Some meshes in the process of the local refinement are presented in Figure \ref{fig: mesh low regularity} from which one can observe that the refinement is around the center of the ellipse where the exact solution is singular.

\begin{figure}[htb]
  \centering
  \includegraphics[width=.4\textwidth]{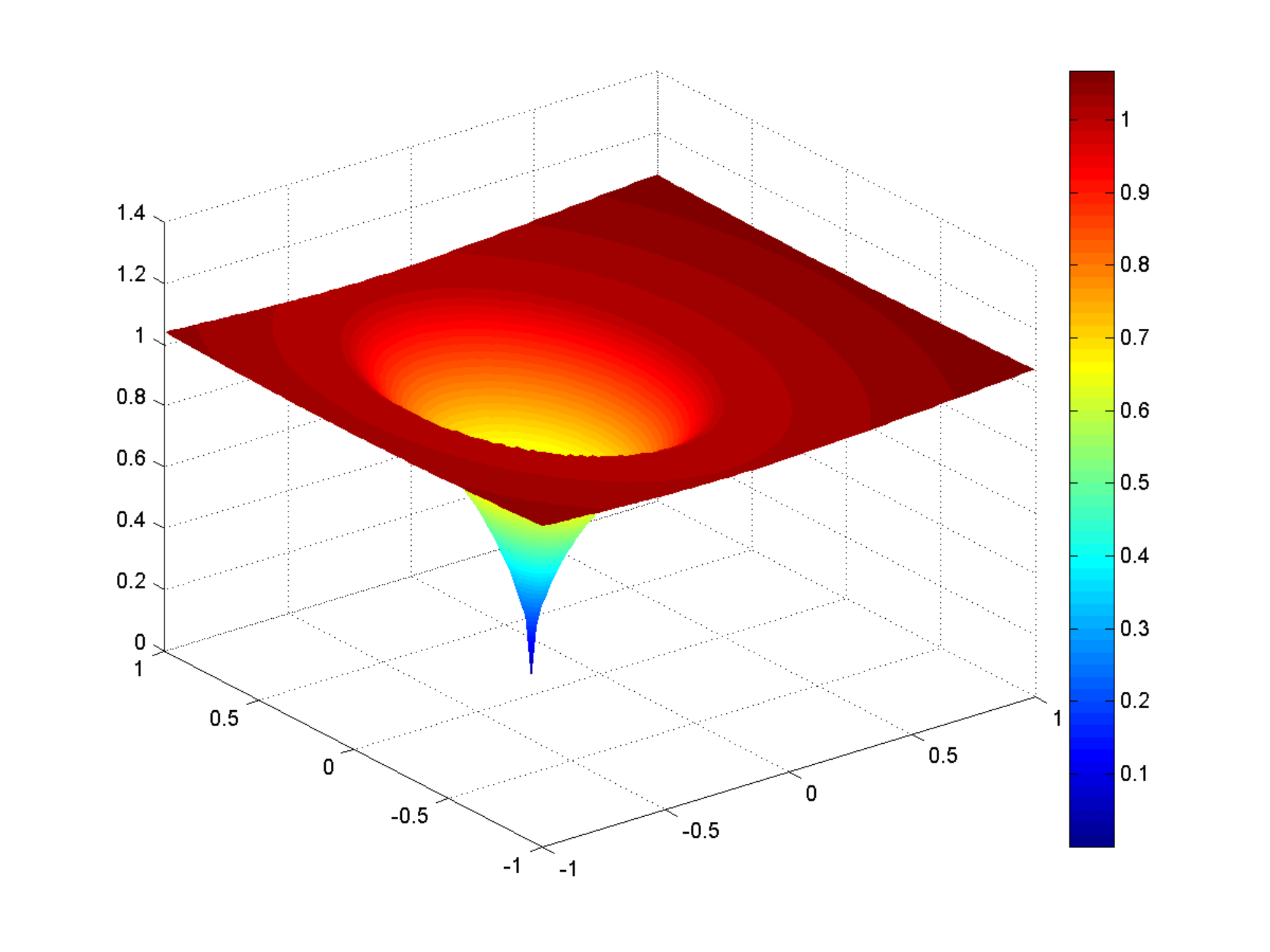}
  \includegraphics[width=.4\textwidth]{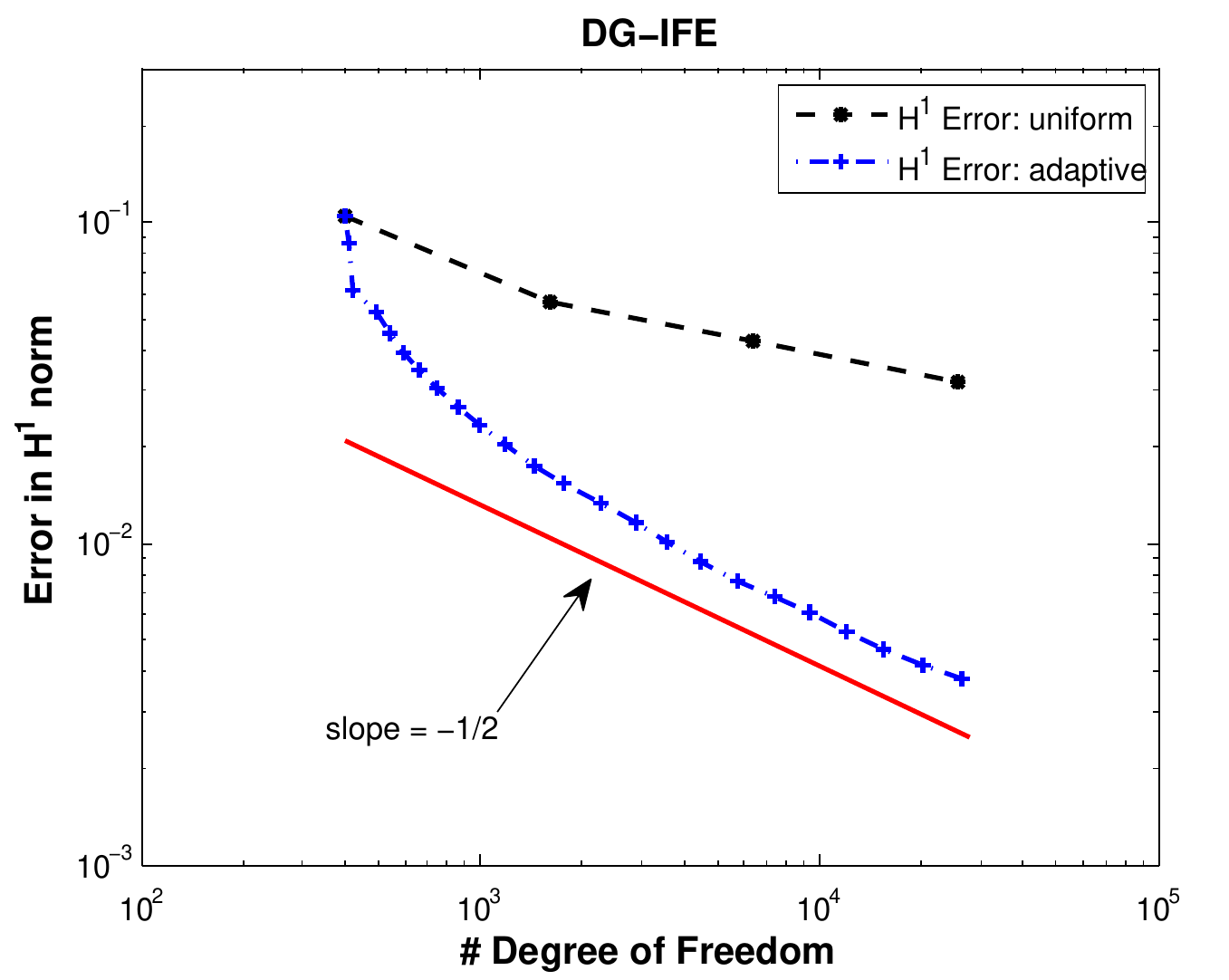}
  \caption{Left: the exact solution $u$ in Example 4. Right: a log-log scale plot of errors of in numerical solutions generated by the adaptive and uniform mesh refinement.
  The red line with a $-0.5$ slope is for the reference of optimal convergence.}
  \label{fig: exactu p=0.5}
\end{figure}

\begin{figure}[ht]
  \centering
  \includegraphics[width=.21\textwidth]{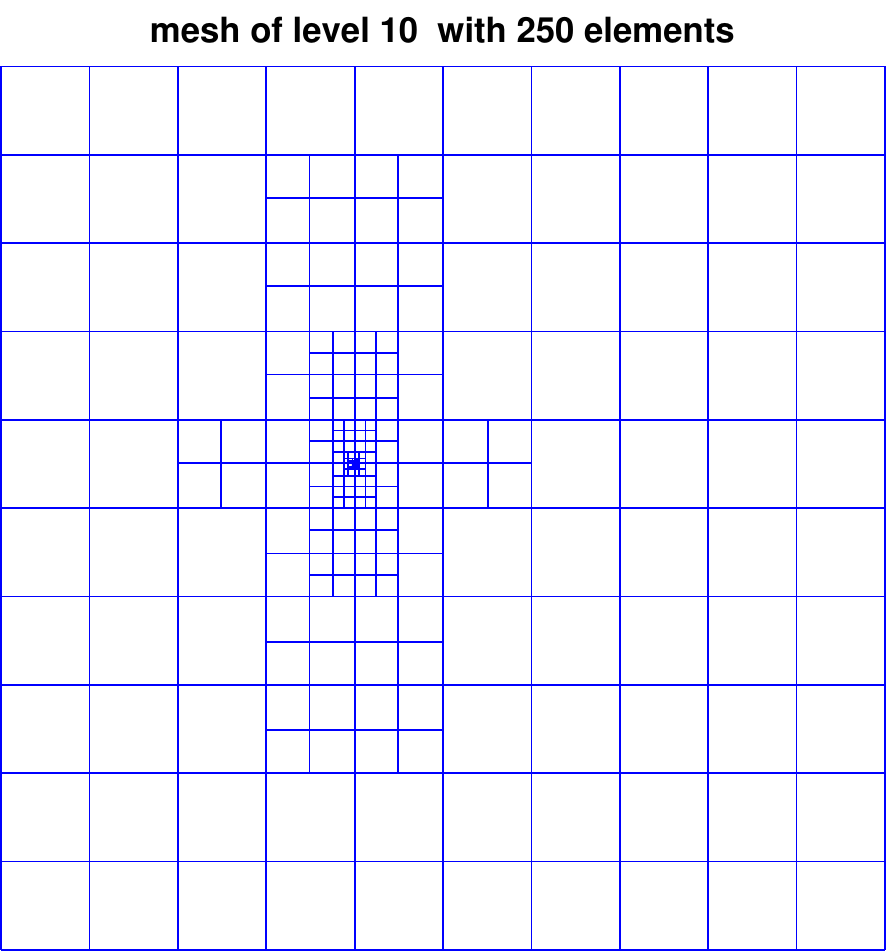}~~
  \includegraphics[width=.21\textwidth]{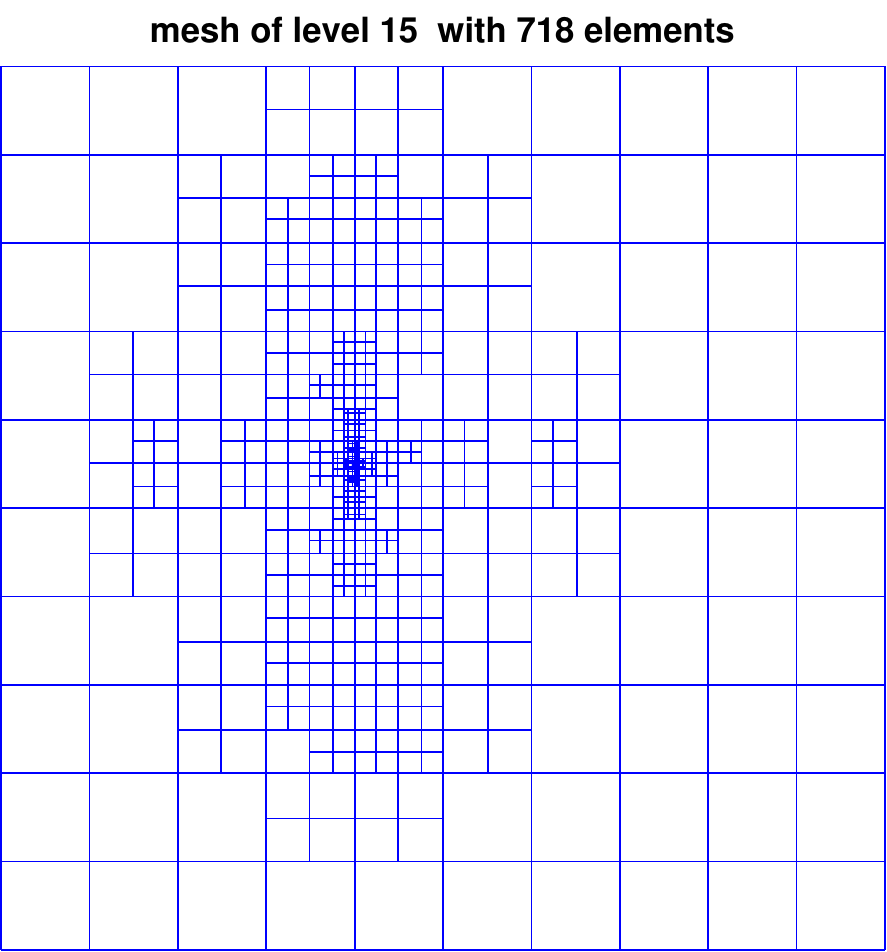}~~
  \includegraphics[width=.21\textwidth]{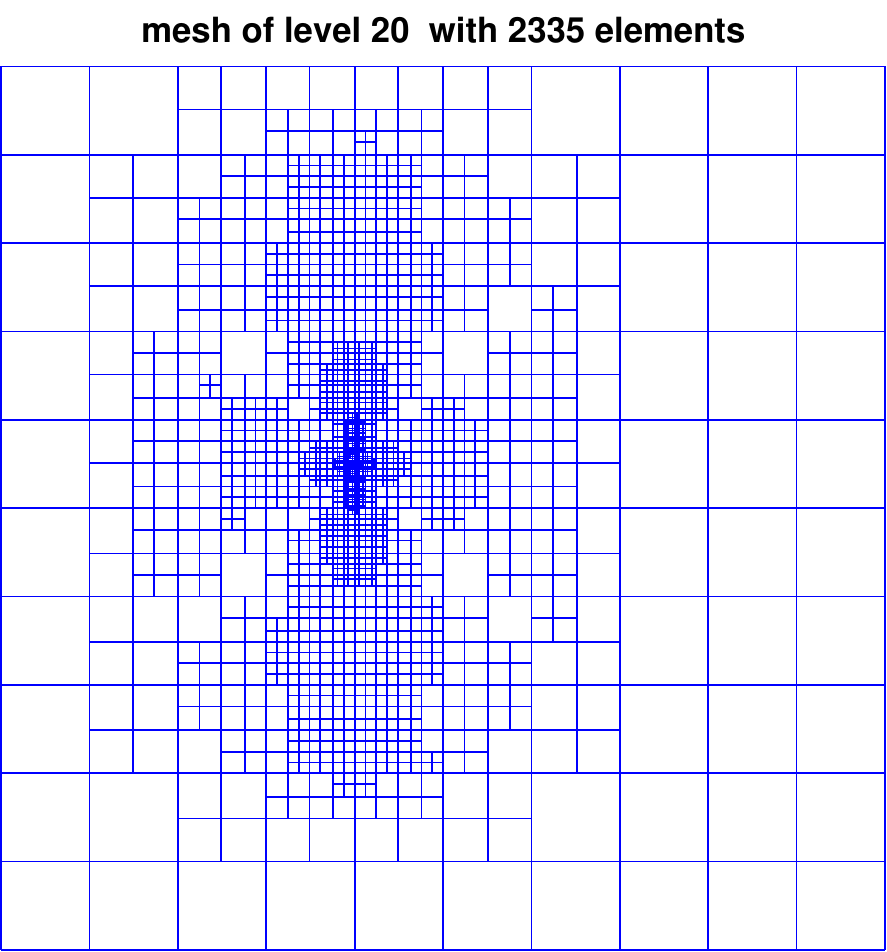}~~
  \includegraphics[width=.21\textwidth]{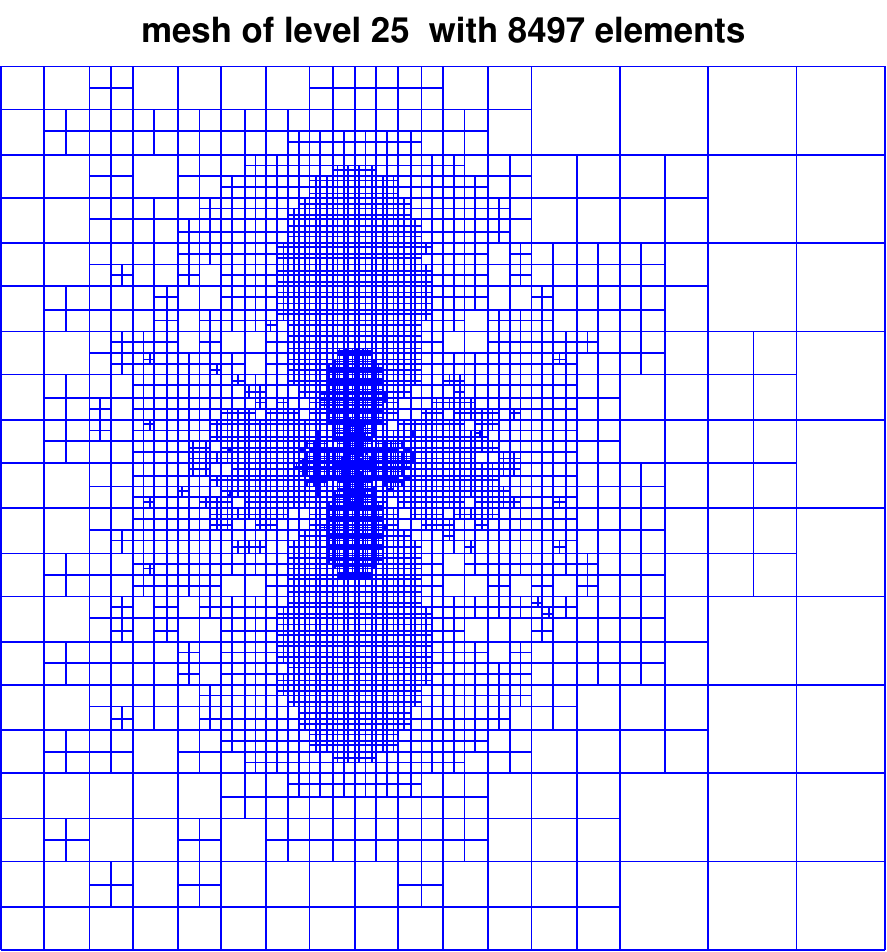}\\
  \caption{Solution meshes: $\mathcal{T}_h^{(10)}$, $\mathcal{T}_h^{(15)}$,$\mathcal{T}_h^{(20)}$, and $\mathcal{T}_h^{(25)}$}
  \label{fig: mesh low regularity}
\end{figure}

\section{Conclusion}
In this article, we establish {\it a priori} error estimates for interior penalty discontinuous Galerkin methods with immersed finite element functions for elliptic interface problem. The method can be used on Cartesian meshes that are independent of the interface. The analysis here shows that the order of convergence of these DG-IFE methods is optimal in the energy norm from the point of the polynomial degree in the finite element spaces. With the enhanced stability, these DG-IFE methods outperform the
the classic Galerkin IFE methods, especially in vicinity of the interface across which the exact solution is usually less smooth. The proposed DG-IFE schemes allow efficient local mesh refinement while preserving the Cartesian structure of meshes provided that {\it a posteriori } error estimators are available.

\subsubsection*{Acknowledgement} The authors would like to thank anonymous referees whose comments and suggestions enhanced the presentation of our research work.

\bibliographystyle{abbrv}

\end{document}